\documentclass[10pt]{article}
\usepackage[preprint]{tmlr}

\usepackage{amsmath,amsfonts,bm}

\def\eqref#1{equation~\ref{#1}}

\def\1{\bm{1}}

\DeclareMathAlphabet{\mathsfit}{\encodingdefault}{\sfdefault}{m}{sl}
\SetMathAlphabet{\mathsfit}{bold}{\encodingdefault}{\sfdefault}{bx}{n}

\DeclareMathOperator*{\argmin}{arg\,min}

\usepackage{amsmath,amssymb,amsthm,mathtools,mathrsfs}
\usepackage{bm}
\usepackage{booktabs}
\usepackage{enumitem}
\usepackage{microtype}
\usepackage{url}
\usepackage{hyperref}
\usepackage[nameinlink,noabbrev]{cleveref}
\hypersetup{hidelinks}

\theoremstyle{plain}
\newtheorem{theorem}{Theorem}
\newtheorem{proposition}{Proposition}
\newtheorem{lemma}{Lemma}
\newtheorem{corollary}{Corollary}
\theoremstyle{definition}
\newtheorem{definition}{Definition}
\newtheorem{assumption}{Assumption}
\theoremstyle{remark}

\newcommand{\SPD}{\mathbb{S}_{++}}

\newcommand{\tr}{\operatorname{tr}}

\newcommand{\rank}{\operatorname{rank}}

\newcommand{\dd}{\mathrm{d}}

\newcommand{\argminop}{\operatorname*{arg\,min}}

\newcommand{\ip}[2]{\left\langle #1,#2\right\rangle}
\newcommand{\norm}[1]{\left\lVert #1\right\rVert}

\newcommand{\fro}[1]{\left\lVert #1\right\rVert_F}
\newcommand{\pospart}[1]{\left[#1\right]_+}

\newcommand{\Xspace}{\mathcal X}
\newcommand{\Yspace}{\mathcal Y}
\newcommand{\Zspace}{\mathcal Z}
\newcommand{\Uspace}{\mathcal U}
\newcommand{\Sigspace}{\mathcal S}
\newcommand{\Bact}{\mathscr B_A}
\newcommand{\Pdet}{\mathscr P_d}
\newcommand{\Pfiber}{\mathscr F_A}
\newcommand{\dai}{d_{\mathrm{AI}}}

\newcommand{\Schur}{\operatorname{Schur}}

\newcommand{\ind}{\iota}

\crefname{assumption}{Assumption}{Assumptions}
\crefname{definition}{Definition}{Definitions}
\crefname{theorem}{Theorem}{Theorems}
\crefname{proposition}{Proposition}{Propositions}
\crefname{lemma}{Lemma}{Lemmas}
\crefname{corollary}{Corollary}{Corollaries}

\title{Optimization Geometrodynamics:\\
Variational Reduction and Interaction Curvature}

\author{\name Zavier Li \email zavierli888@gmail.com\\
      \addr Xidian University\\
      \addr Xi'an, China}

\def\month{07}
\def\year{2026}
\def\openreview{\url{}}
\newcommand{\informationgeometrycite}{%
  \citep{li2026informationgeometry}}

\begin{document}
\maketitle

\begin{abstract}
Adaptive optimizers carry hidden states that change how visible gradients
become parameter motion.  We develop optimization geometrodynamics as a
variational theory of this hidden geometry.  Infimal pushforward eliminates
all hidden states realizing the same action and composes across optimizer
hierarchies.  Smooth nondegeneracy gives the hidden susceptibility and the
Schur-complement curvature seen after relaxation.  In particular, affine
pre-reduction perturbations induce the negative-semidefinite interaction
curvature \(-G^*H^{-1}G\), whose mixed entries integrate to finite mechanism
contrasts.

The main realization is the determinant-one affine-invariant SPD action map
\(P\mapsto PA\).  We prove a global analytic bundle with closed totally
geodesic fibers and a unique analytic nearest-controller section.  More
importantly, a general strongly-convex fiber theorem and its explicit action
residual yield a solver that converges globally and linearly from every
feasible initializer, with
nonasymptotic value, controller-distance, and residual bounds and observable
posterior stopping certificates.  The solver remains globally convergent
conditional on supplied rigorous residual-error and radius majorants.  Its dense spectral kernel is
restricted to an active subspace of dimension \(r\le2m\), giving an explicit
spectral-arithmetic operation bound after one-time input reduction.  For nested
shape-normalized quadratic actions, canonical multi-secant projections obey
an exact CAT(0) Pythagorean decrease and recover the determinant-one inverse
Hessian shape at the sharp rank threshold \(d-1\); this conclusion requires a
known scalar gauge.  Together these results turn the action bundle from a
geometric characterization into an exact iteration with posterior certificates
and a finite-identification theory.
\end{abstract}

\section{Introduction}
\label{sec:introduction}

Adaptive optimizers use hidden positive operators---moments, curvature
surrogates, Kronecker factors, or tensor preconditioners---to turn visible
gradients into motion \citep{duchi2011adagrad,kingma2014adam,
martens2015kfac,gupta2018shampoo,morwani2024shampoo,jordan2024muon,
essentialai2025muon,crawshaw2025muonvariants}.  Writing
\(\dot\theta=-\operatorname{grad}_{g_t}f\) does not select which hidden metric
should represent an observed action when many metrics realize it.  We study
that ambiguity as a variational reduction problem.

For hidden state \(x\), visible action \(y=\pi(x)\), and deformation energy
\(E\), the infimal pushforward
\[
  (\pi_!E)(y)=\inf_{\pi(x)=y}E(x)
\]
eliminates hidden realizations and composes across nested maps.  Convex-fiber
regularity selects a canonical lift; smooth nondegeneracy gives
\[
  D\sigma_\star=-E_{\sigma\sigma}^{-1}E_{\sigma y},
  \qquad
  D^2(\pi_!E)
  =E_{yy}-E_{y\sigma}E_{\sigma\sigma}^{-1}E_{\sigma y}.
\]
When mechanism amplitudes enter affinely before reduction, the corresponding
interaction curvature is
\[
  D^2_{uu}(\pi_!E)=-G^*H^{-1}G\preceq0,
\]
and its mixed entries integrate to finite contrasts for the specified
mechanism protocol.

Our main realization is the determinant-one SPD action map
\[
  \pi_A(P)=PA.
\]
For full-column-rank \(A\), every admissible fiber has a global analytic
parameterization by a smaller determinant-one SPD factor and is closed and
totally geodesic under AIRM.  Projection of the identity therefore gives a
unique analytic canonical controller.  A general strongly-convex fiber
engine connects this selection to computation, while the action bundle makes
its logarithmic residual explicit.  A sublevel first-return argument yields
the sharp radius majorant \(L_0=\psi(D_0/\sqrt2)\), nonasymptotic iteration
bounds obtained by exact inversion of the linear rate, and posterior
controller and objective certificates.  Relative matrix-function error
remains globally convergent when a numerical routine supplies rigorous error
and radius majorants.

The solver has an exact active realization on
\(\operatorname{span}(\operatorname{col}A,\operatorname{col}B)\), of dimension
\(r\le2m\).  It replaces the full spectral kernel by \(r\times r\) operations
and gives a strict dimension reduction when \(r<d\).  Nested shape-normalized
actions then satisfy a CAT(0)
Pythagorean law and identify \(c_HH^{-1}\) at the sharp rank \(d-1\), provided
the scalar gauge \(c_H=(\det H)^{1/d}\) is known.

\paragraph{Contributions.}
\begin{enumerate}[leftmargin=*,itemsep=0.2em]
  \item \textbf{Reduction and interaction curvature.}
  Associative hidden-state elimination yields the canonical response and Schur
  effective Hessian; action specialization identifies the shared inverse
  vertical Hessian behind interaction and controller susceptibility.

  \item \textbf{Global action geometry and posterior-certified computation.}
  We trivialize \(P\mapsto PA\), prove fiber total geodesy and analytic
  selection, and derive a global exact solver with sharp sublevel and
  posterior constants, plus a conditional inexact analogue.

  \item \textbf{Active spectral cost and finite identification.}
  The active core has order \(r\le2m\), costs \(O(r^3)\) per step after
  \(O(dm^2+r^3)\) preprocessing, and applies in \(O(dr+r^2)\).  Nested
  projections recover determinant-one inverse shape after
  \(\lceil(d-1)/b\rceil\) independent width-\(b\) rounds.

\end{enumerate}

The general reduction, Schur, and strongly-convex descent ingredients are
classical.  The action-specific contribution is their global closure through
the analytic bundle, explicit residual, active realization, conditional
error propagation, and sharp normalized recovery threshold.  The action
specialization of the interaction tensor identifies the same vertical
Hessian that governs canonical response, so these results form one
variational-reduction master chain.

\paragraph{Organization.}
\Cref{sec:reduction,sec:interaction} develop reduction and interaction;
\cref{sec:action-bundle,sec:certified-solver} give the action geometry and
solver; \cref{sec:multisecant-recovery} gives finite recovery.  Appendices
contain complete proofs and a closest-work comparison.

\section{Related Work}
\label{sec:related-work}

\paragraph{Adaptive geometry and variational reduction.}
Natural gradient, K-FAC, AdaGrad, Adam, Shampoo, SOAP, and related methods use
Fisher, moment, Kronecker, or tensor states as changing positive update
operators \citep{amari1998natural,martens2015kfac,duchi2011adagrad,
kingma2014adam,gupta2018shampoo,morwani2024shampoo,vyas2025soap}.
Riemannian optimization provides the coordinate-free gradient language
\citep{absil2008optimization,boumal2023introduction}.  Partial minimization,
marginal functions, envelope sensitivity, and variable projection are
classical \citep{rockafellar1970convex,rockafellar1998variational,
bonnans2000perturbation,golub1973variableprojection}; we do not claim the
Schur-complement identity itself as new.
Global first-order complexity for geodesically convex objectives on Hadamard
manifolds is established in \citet{zhang2016firstorder}; metric projection and
convex optimization in Hadamard spaces are treated systematically by
\citet{bacak2014convex}.  We therefore do not claim the strongly-convex
gradient-descent template in \cref{thm:canonical-fiber-engine}.  Its role is
to expose a current-sublevel first-return argument in exactly the form needed
by the action problem.  The action-specific claims are the global fiber
coordinates, exact logarithmic residual, sharp action-family radius majorant,
active equivalence, and posterior controller certificates.

\paragraph{AIRM approximation and totally geodesic submanifolds.}
The affine-invariant SPD geometry is classical
\citep{bhatia2007positive,pennec2006riemannian}.  AIRM approximation from
special sets, best approximation in geodesic SPD submanifolds, and general
total-geodesy criteria and projections are studied by
\citet{bhatia2014approximation}, \citet{lim2004best}, and
\citet{tumpach2024totally}; conic geodesic optimization supplies a broader
algorithmic SPD setting \citep{sra2015conic}.  Algebraic identities and
definiteness conditions for Hermitian solutions of \(AX=B\) are studied by
\citet{tian2013hermitian}; we do not claim the bare feasibility criterion or
block-Schur solution algebra as new.  Once a fiber is known to be
closed and totally geodesic, the projection theory accounts for its unique
AIRM projection.  \citet{lim2004best} additionally gives explicit formulas
for selected involutive block models, including low-dimensional block cases.
Our specialization treats the action-indexed determinant-one fiber in every
dimension and adds its global AIRM bundle, analytic section, exact
residual-gradient map, sharp current-sublevel majorant, posterior certificates,
and active spectral realization.  The conditional inexact result propagates
rigorous numerical error and radius majorants; it does not construct those
matrix-function certificates.

\paragraph{Secant updates and completion.}
Inverse-Hessian multi-secant equations have the form \(PY=S\)
\citep{schnabel1983multiple}; compatibility with positive definiteness and
modern block or grouped updates have been studied separately
\citep{passy1984secant,gao2018block,boutet2022grouped}.  Classical least-change
and variational BFGS/DFP
updates use weighted norms or log-determinant divergences
\citep{dennis1979leastchange,fletcher1991variational}; Bregman and weighted
secant extensions use related selection principles
\citep{kanamori2013bregman,gratton2015weighted}.  Bounded and sparse
quasi-Newton updates and positive-definite completion are treated in
\citet{calamai1987bounds,yamashita2007sparse,dai2014sparse,grone1984positive}.

The present action-specific contribution starts before the general projection
principle: it gives a global analytic trivialization of \(P\mapsto PA\),
identifies its determinant-one fibers and proves their total geodesy, then
derives a globally convergent residual solver and the sharp completion rank
for shape-normalized nested actions.  Classical multi-secant observations do
not include the determinant scale used by the \(d-1\) recovery theorem; the
manuscript states that scalar-information requirement explicitly.
Condition-constrained SPD approximation and preconditioning have separate literatures
\citep{tanaka2014condition,marechal2009optimizing,lu2011minimizing,
gao2023scalable,qu2024optimal,dogan2025geodesicpreconditioning}.
A theorem-level comparison is given in \cref{app:closest-comparison}.

\paragraph{Scope.}
The companion information-compression map
\(P\mapsto A^\top PA\) has different fibers
\informationgeometrycite.  The present paper is self-contained
and treats the full-SPD exact-action problem.  Structured controller families,
noisy secants, unknown scalar gauges, historical state estimation, and regret
require additional analyses.

\section{Variational Reduction of Hidden Geometry}
\label{sec:reduction}

This section isolates the operation that turns hidden optimizer geometry into
a visible action theory.  No smoothness is needed for the algebraic layer;
geometry enters when one asks for a canonical representative of each fiber.

\subsection{Infimal pushforward}

Let \(\pi:\Xspace\to\Yspace\) be an arbitrary observation map and let
\(E:\Xspace\to(-\infty,+\infty]\) be an extended energy that never takes
\(-\infty\).  We use \(\inf\varnothing=+\infty\).

\begin{definition}[Infimal pushforward]
\label{def:infimal-pushforward}
The infimal pushforward of \(E\) along \(\pi\) is
\begin{equation}
  (\pi_!E)(y)
  :=\inf_{\pi(x)=y}E(x).
  \label{eq:infimal-pushforward}
\end{equation}
Its effective domain is the set of visible actions that admit a finite-energy
hidden realization.
\end{definition}

\begin{theorem}[Exact reduction and composition]
\label{thm:algebraic-reduction}
Let \(\rho:\Yspace\to\Zspace\) be another map.  Then
\begin{equation}
  (\rho\circ\pi)_!E=\rho_!(\pi_!E).
  \label{eq:pushforward-composition}
\end{equation}
For every \(\ell:\Yspace\to(-\infty,+\infty]\) for which the sums are
defined,
\begin{equation}
  \inf_{x\in\Xspace}\{E(x)+\ell(\pi(x))\}
  =
  \inf_{y\in\Yspace}\{(\pi_!E)(y)+\ell(y)\}.
  \label{eq:task-reduction}
\end{equation}
Moreover, for a visible potential \(\varphi\) and a constraint
\(\mathcal C\subseteq\Xspace\),
\begin{align}
  \pi_!(E+\varphi\circ\pi)&=\pi_!E+\varphi,
  \label{eq:base-potential-law}\\
  \pi_!(E+\ind_{\mathcal C})(y)
  &=\inf_{\substack{\pi(x)=y\\x\in\mathcal C}}E(x).
  \label{eq:constraint-law}
\end{align}
\end{theorem}

The theorem is a typed regrouping law.  Its value is compositional: a
hierarchical optimizer may eliminate full geometry, structured factors,
visible action scale, and normalized direction in stages without changing the
final reduced energy.  It does not by itself guarantee a minimizer, regularity,
or a useful geometry on the visible space.

\subsection{Canonical minimizer sections}

We next give sufficient conditions for the reduced value to select one hidden
state.  Let \((\Xspace,d)\) be a proper Hadamard space and write
\(\mathcal F_y=\pi^{-1}(y)\).

\begin{assumption}[Convex-fiber regularity]
\label{ass:convex-fiber}
Every nonempty fiber \(\mathcal F_y\) is closed and geodesically convex.
The energy \(E\) is lower semicontinuous, and its restriction to each fiber is
coercive and \(\mu\)-strongly geodesically convex for a common \(\mu>0\).
\end{assumption}

\begin{theorem}[Canonical hidden geometry]
\label{thm:canonical-lift}
Under \cref{ass:convex-fiber}, every \(y\) with finite \((\pi_!E)(y)\) has a
unique minimizer
\begin{equation}
  s_E(y)=\argminop_{\pi(x)=y}E(x),
  \qquad
  E(s_E(y))=(\pi_!E)(y).
  \label{eq:canonical-lift}
\end{equation}
If a group acts isometrically on \(\Xspace\), acts on \(\Yspace\), and leaves
\(E\) and \(\pi\) equivariant, then
\begin{equation}
  s_E(gy)=g s_E(y).
  \label{eq:canonical-equivariance}
\end{equation}
\end{theorem}

The minimizer section is energy-dependent.  It should be distinguished from
the shortest section of a metric submetry, which is selected by distance alone.
The two coincide when \(E(x)=\tfrac12d(x_0,x)^2\) and the relevant shortest
completion is unique.

\section{Hidden Response and Interaction Curvature}
\label{sec:interaction}

The algebraic reduction becomes a differential theory when the canonical
minimizer varies smoothly.  We state the result in local product coordinates;
at a fiber critical point the resulting reduced quadratic form is intrinsic.

\subsection{Smooth response and effective curvature}

Let \(z=(y,u)\) range over a finite-dimensional manifold locally represented
in Euclidean coordinates, and let \(\Sigspace\) be a finite-dimensional
smooth manifold with local hidden coordinate \(\sigma\).  Consider
\[
  E:(z,\sigma)\longmapsto E(z,\sigma)\in\mathbb R.
\]

\begin{assumption}[Smooth stable reduction]
\label{ass:smooth-reduction}
For some \(r\ge2\), the energy is \(C^{r+1}\).  It is locally uniformly
inf-compact in \(\sigma\): for every compact parameter set \(K\) and every
\(c\in\mathbb R\), the union of the fiber sublevel sets
\[
  \{\sigma:\exists z\in K,\ E(z,\sigma)\le c\}
\]
is relatively compact.  Every \(z\) has a unique interior minimizer
\(\sigma_\star(z)\), and
\[
  H(z):=E_{\sigma\sigma}(z,\sigma_\star(z))
\]
is positive definite.
\end{assumption}

\begin{theorem}[Hidden response and Schur effective curvature]
\label{thm:response-schur}
Under \cref{ass:smooth-reduction}, the minimizer section and reduced energy
\[
  \sigma_\star(z)=\argmin_\sigma E(z,\sigma),
  \qquad
  \bar E(z)=E(z,\sigma_\star(z))
\]
are \(C^r\).  At the canonical lift,
\begin{align}
  D_z\sigma_\star&=-H^{-1}E_{\sigma z},
  \label{eq:hidden-response}\\
  D_z\bar E&=E_z,
  \label{eq:envelope}\\
  D^2_{zz}\bar E
  &=E_{zz}-E_{z\sigma}H^{-1}E_{\sigma z}.
  \label{eq:schur-effective-hessian}
\end{align}
If \(\sigma=(\alpha,\beta)\), let \(\mathbb H\) denote the full Hessian in
\((z,\alpha,\beta)\) at the canonical lift.  When its joint hidden block is
positive definite, sequential and joint second-order elimination agree:
\begin{equation}
  \Schur_{(\alpha,\beta)}\mathbb H
  =\Schur_\alpha\bigl(\Schur_\beta\mathbb H\bigr).
  \label{eq:schur-associativity}
\end{equation}
\end{theorem}

The response formula is familiar from parametric optimization and variable
projection \citep{golub1973variableprojection,bonnans2000perturbation}.  Here
it supplies the local law inside a globally defined geometric action bundle
and composes with the optimizer hierarchy in \cref{thm:algebraic-reduction}.

\subsection{Affine perturbations before reduction}

Let \(u\in\Uspace\subset\mathbb R^p\) denote continuous mechanism amplitudes.
Suppose
\begin{equation}
  E(y,\sigma,u)
  =E_0(y,\sigma)+\sum_{i=1}^p u_i\Delta_i(y,\sigma).
  \label{eq:affine-interventions}
\end{equation}
At the canonical lift define
\begin{equation}
  G:T_u\Uspace\to T_\sigma^*\Sigspace,
  \qquad
  Ga=\sum_{i=1}^p a_iD_\sigma\Delta_i.
  \label{eq:interaction-map}
\end{equation}

\begin{theorem}[Reduction-induced interaction curvature]
\label{thm:interaction-curvature}
Under \cref{ass:smooth-reduction,eq:affine-interventions},
\begin{equation}
  \boxed{
  D^2_{uu}\bar E=-G^*H^{-1}G\preceq0.}
  \label{eq:interaction-curvature}
\end{equation}
Equivalently,
\begin{equation}
  \partial^2_{u_i u_j}\bar E
  =-
  \left\langle
  D_\sigma\Delta_i,
  H^{-1}D_\sigma\Delta_j
  \right\rangle.
  \label{eq:pair-interaction-curvature}
\end{equation}
Thus the self-curvature of every mechanism is nonpositive.  A pointwise mixed
interaction vanishes exactly when the corresponding vertical perturbations
are orthogonal in the inverse-Hessian metric.  Mixed entries may have either
sign.
\end{theorem}

Even without smoothness, \(u\mapsto\bar E(y,u)\) is concave wherever finite,
because it is a pointwise infimum of affine functions.  Smoothness identifies
the exact curvature responsible for that concavity.

The affine hypothesis is essential for the sign.  With general smooth
\(u\)-dependence, \cref{eq:schur-effective-hessian} gives
\begin{equation}
  D^2_{uu}\bar E
  =E_{uu}-E_{u\sigma}H^{-1}E_{\sigma u},
  \label{eq:nonaffine-interaction-curvature}
\end{equation}
so the direct curvature \(E_{uu}\) need not be negative semidefinite.  The
quadratic form is covariant under a change \(u=T\tilde u\), transforming as
\(D^2_{\tilde u\tilde u}\bar E=T^*D^2_{uu}\bar E\,T\).  Individual mixed
entries therefore have meaning only after the mechanism basis, direction,
and units are fixed; all comparisons in this paper refer to the specified
amplitude protocol.

Factorial finite differences are iterated integrals of the corresponding
mixed derivatives; \cref{cor:factorial-contrast} states the general identity.
In particular,
\begin{equation}
  \delta_{\{i,j\}}\bar E
  =-
  \int_0^1\!\int_0^1
  \left\langle D_\sigma\Delta_i,
  H^{-1}D_\sigma\Delta_j\right\rangle
  \,\dd u_i\,\dd u_j,
  \label{eq:pair-factorial-integral}
\end{equation}
where all quantities are evaluated along the canonical minimizer with the
other amplitudes fixed at zero.
The integral may vanish by cancellation even when pointwise interaction does
not.  Reduction must precede finite differencing; the counterexample and the
exact commuting special case are given in \cref{app:finite-contrasts}.

\subsection{Action-fiber specialization}

The global action chart makes the same response tensor computational.  For
fixed \(B\in\Bact\), perturb the canonical fiber objective by
\begin{equation}
  F_B(\Sigma,u)
  =f_B(\Sigma)+\sum_{i=1}^p u_i\Delta_i(B,\Sigma),
  \qquad \Sigma\in\mathscr P_{d-m}.
  \label{eq:perturbed-action-fiber}
\end{equation}
Let \(H_B=\operatorname{Hess}f_B(\Sigma_\star)\) and
\(G_Ba=\sum_i a_iD_\Sigma\Delta_i(B,\Sigma_\star)\).  Strong convexity of
the action objective gives \(H_B\succeq\operatorname{Id}\) as a self-adjoint
operator with respect to the AIRM metric.  For an intrinsic perturbation,
we require the natural complement-frame covariance
\(\Delta_{i,VR}(B,R^\top\Sigma R)=\Delta_{i,V}(B,\Sigma)\); otherwise the
chosen \(\Delta_i\) deliberately selects a hidden frame.

\begin{corollary}[Interaction response of the canonical action controller]
\label{cor:action-interaction-response}
Assume the \(\Delta_i\) are \(C^2\) near \(\Sigma_\star\).  For sufficiently
small \(u\), there is a unique critical point \(\Sigma_\star(B,u)\) near
\(\Sigma_\star\); it is a strict local minimizer.  Define the corresponding
local value branch by
\(\bar F_B^{\rm loc}(u)=F_B(\Sigma_\star(B,u),u)\).  Then
\begin{equation}
  D_u\Sigma_\star(B,0)=-H_B^{-1}G_B,
  \qquad
  D^2_{uu}\bar F_B^{\rm loc}(0)
  =-G_B^*H_B^{-1}G_B.
  \label{eq:action-interaction-response}
\end{equation}
The same inverse vertical Hessian \(H_B^{-1}\) governs the hidden-coordinate
response to the visible action \(B\).  The full section derivative also
contains the chart's direct term:
\[
  D_Bs_A=D_B\Phi_A+D_\Sigma\Phi_A[D_B\Sigma_\star].
\]
Thus interaction curvature, action susceptibility, and the residual solver
are local response, base response, and computation for one canonical fiber
problem.
\end{corollary}

This response statement is local in the perturbation.  A global perturbed
minimum requires separate coercivity or growth control on the \(\Delta_i\).

\begin{proof}
At \((\Sigma_\star,0)\), the derivative in \(\Sigma\) of the stationarity
equation is the invertible operator \(H_B\).  The implicit-function theorem
therefore gives the unique nearby critical branch; continuity of its Hessian
makes it a strict local minimum after shrinking the \(u\)-neighborhood.
Differentiating stationarity and then the local value branch gives
\cref{eq:action-interaction-response}.  The covariance law makes this branch
independent of complement coordinates.  Differentiation with respect to \(B\)
uses the same stationarity equation with \(E_{\Sigma B}\) in place of \(G_B\).
\end{proof}

\section{Global SPD Action Geometry}
\label{sec:action-bundle}

We now realize the abstract reduction globally.  Let
\[
  \Pdet=\{P\in\SPD^d:\det P=1\},
  \qquad
  r(P)=\dai(I,P)=\fro{\log P}.
\]
Fix a full-column-rank \(A\in\mathbb R^{d\times m}\), with \(1\le m<d\)
and \(n=d-m\).  The visible action map and its admissible base are
\begin{equation}
  \pi_A:\Pdet\to\mathbb R^{d\times m},
  \qquad
  \pi_A(P)=PA,
  \qquad
  \Bact=\{B:A^\top B\in\SPD^m\}.
  \label{eq:action-map-base}
\end{equation}
The condition \(A^\top B\in\SPD^m\) includes symmetry and is necessary and
sufficient for an SPD solution of \(PA=B\).

Choose a thin QR factorization and orthogonal completion
\[
  A=UC,
  \qquad U^\top U=I_m,
  \qquad Q=[U,V]\in O(d).
\]
For \(B\in\Bact\), define
\begin{align}
  M(B)&=C^{-\top}A^\top BC^{-1},
  &N(B)&=V^\top BC^{-1},
  \label{eq:action-coordinates}\\
  T(B)&=N(B)M(B)^{-1},
  &L(B)&=\begin{pmatrix}I_m&0\\T(B)&I_n\end{pmatrix},
  \label{eq:action-triangular}\\
  \rho(B)&=(\det M(B))^{-1/n},
  &\mathscr P_n&=\{\Sigma\in\SPD^n:\det\Sigma=1\}.
  \label{eq:action-fiber-coordinates}
\end{align}

\begin{theorem}[Global action bundle and canonical controller]
\label{thm:action-bundle}
The map
\begin{equation}
  \Phi_A:\Bact\times\mathscr P_n\to\Pdet,
  \qquad
  \Phi_A(B,\Sigma)
  =QL(B)
  \begin{pmatrix}M(B)&0\\0&\rho(B)\Sigma\end{pmatrix}
  L(B)^\top Q^\top
  \label{eq:action-bundle-map}
\end{equation}
is a global real-analytic diffeomorphism satisfying
\(\pi_A(\Phi_A(B,\Sigma))=B\).  Consequently,
\begin{equation}
  \Pdet\cong\Bact\times\mathscr P_{d-m},
  \qquad
  \Bact\cong\SPD^m\times\mathbb R^{n\times m}.
  \label{eq:global-action-product}
\end{equation}
For every \(B\), the fiber
\[
  \Pfiber(B)=\{P\in\Pdet:PA=B\}
\]
is nonempty, closed, totally geodesic, and AIRM-isometric to
\(\mathscr P_n\).

Every fiber has a unique nearest controller
\begin{equation}
  s_A(B)=\argminop_{P\in\Pfiber(B)}r(P),
  \qquad
  \mathfrak C_A(B)=r(s_A(B)).
  \label{eq:canonical-controller}
\end{equation}
The section \(s_A\) and squared cost \(\mathfrak C_A^2\) are real analytic.
They are intrinsic under action-basis changes and orthogonal ambient changes:
\begin{equation}
  s_{AD}(BD)=s_A(B),
  \qquad
  s_{RA}(RB)=R s_A(B)R^\top
  \label{eq:controller-equivariance}
\end{equation}
for \(D\in\operatorname{GL}_m\) and \(R\in O(d)\).
\end{theorem}

The identity is a reference controller, not an essential restriction.  For
any \(P_{\rm ref}\in\Pdet\), congruence by \(P_{\rm ref}^{-1/2}\) sends the
nearest-completion problem from \(P_{\rm ref}\) to the identity-reference
problem with transformed action
\(\widetilde A=P_{\rm ref}^{1/2}A\) and
\(\widetilde B=P_{\rm ref}^{-1/2}B\).  Thus the bundle and solver apply to an
arbitrary determinant-one baseline without changing their geometry.

The proof responsibility in this theorem is worth separating.  General
Hadamard and SPD-submanifold theory supplies the unique nearest point after a
fiber has been shown closed and totally geodesic
\citep{lim2004best,tumpach2024totally}.  The action-specific results are the
explicit global trivialization \(\Phi_A\), the identification and total
geodesy of every \(PA=B\) fiber, analytic dependence on \(B\), the residual
equation below, and the solver and active reduction in
\cref{sec:certified-solver}.  Projection uniqueness by itself is not claimed
as new.

The action constraint gives a complete inverse-Hessian multi-secant instance
of the abstract theory.  For \(Y,S\in\mathbb R^{d\times m}\), take
\[
  \Xspace=\Pdet,\qquad
  \pi_Y(P)=PY,\qquad
  E(P)=\tfrac12\dai(I,P)^2,\qquad
  y=S.
\]
Then \(Y^\top S\in\SPD^m\) is exactly the admissibility condition and the
canonical lift \(s_E(S)\) is \(s_Y(S)\) from
\cref{eq:canonical-controller}.  Thus the theorem identifies every hidden SPD
completion and selects the AIRM least-deformation controller for an admissible
multi-secant system.  Classical quasi-Newton updates select by other norms or
divergences
\citep{dennis1979leastchange,fletcher1991variational,kanamori2013bregman}; the
global bundle, total-geodesy, AIRM certificate, and active reduction are the
additional objects established here.

The theorem gives a global hidden coordinate \(\Sigma\).  It also supplies a
residual characterization and an a posteriori certificate.  For fixed \(B\),
abbreviate \(M=M(B)\), \(L=L(B)\), and \(\rho=\rho(B)\), and set
\[
  R_0=L^{-1}L^{-\top},
  \qquad
  \mathcal Q_\Sigma=\operatorname{diag}(M,\rho\Sigma),
\]
\begin{equation}
  \mathcal R(\Sigma)
  =\log(\mathcal Q_\Sigma^{-1/2}R_0\mathcal Q_\Sigma^{-1/2}).
  \label{eq:action-residual}
\end{equation}
If \(X^\circ=X-\tr(X)I_n/n\), the canonical hidden coordinate is the unique
solution of
\begin{equation}
  \mathcal R_{22}(\Sigma_\star)^\circ=0.
  \label{eq:action-normal-equation}
\end{equation}
For every trial \(\Sigma\in\mathscr P_n\),
\begin{align}
  \dai(\Sigma,\Sigma_\star)
  &\le\fro{\mathcal R_{22}(\Sigma)^\circ},
  \label{eq:action-distance-certificate}\\
  0\le f_B(\Sigma)-f_B(\Sigma_\star)
  &\le\frac12\fro{\mathcal R_{22}(\Sigma)^\circ}^2,
  \label{eq:action-value-certificate}
\end{align}
where \(f_B(\Sigma)=\tfrac12\dai(R_0,\mathcal Q_\Sigma)^2\).

\subsection{Active reduction}

Let
\[
  W=\operatorname{span}(\operatorname{col}A,\operatorname{col}B),
  \qquad r=\dim W\le2m,
  \qquad q=d-r.
\]
If \(q>0\), the canonical controller has the form
\begin{equation}
  s_A(B)=H_\star\oplus(\det H_\star)^{-1/q}I_q,
  \label{eq:active-controller}
\end{equation}
where \(H_\star\in\operatorname{SPD}(W)\) uniquely minimizes
\begin{equation}
  \min_{\substack{H\in\operatorname{SPD}(W)\\HA=B}}
  \left\{\fro{\log H}^2+\frac{(\log\det H)^2}{q}\right\}.
  \label{eq:active-controller-problem}
\end{equation}
Thus the matrix variable is supported on an ambient subspace of dimension
\(r\le2m\), rather than \(d\).  It still has \(O(r^2)\) scalar degrees of
freedom; the claim is an active-subspace reduction, not a \(2m\)-variable
parameterization.
The next section turns the normal equation into a globally convergent
residual-gradient iteration and proves that the same active symmetry is
preserved at every step.

\subsection{Scale and reference}

The determinant constraint fixes unit-volume shape; it does not create a
dynamical gauge, because rescaling changes \(PA\) and the update time scale.
Dropping the constraint replaces \(\mathscr P_n\) by \(\SPD^n\).  The
complement volume keeps the admissible base \(\Bact\) unchanged, while the
feasible fiber and canonical section change.  The compression
\(P\mapsto A^\top PA\) has a different fiber
\informationgeometrycite; here the complete action \(PA\) is
fixed.

\section{A Globally Certified Reduced Solver}
\label{sec:certified-solver}

We first isolate the computation principle behind canonical reduction.  This
also clarifies which part is general Hadamard optimization and which part is
specific to the SPD action bundle.

\begin{theorem}[Canonical strongly-convex fiber engine]
\label{thm:canonical-fiber-engine}
Let \(\mathcal C\) be a closed geodesically convex totally geodesic
submanifold of a finite-dimensional Hadamard manifold.  Suppose
\(F:\mathcal C\to\mathbb R\) is \(C^2\), coercive, and
\(\mu\)-strongly geodesically convex.  If
\[
  \operatorname{Hess}F\preceq L_0g
  \quad\text{on}\quad
  \{x:F(x)\le F(x_0)\},
\]
then \(F\) has a unique minimizer \(x_\star\), every complete update segment
of
\begin{equation}
  x_{k+1}
  =\operatorname{Exp}_{x_k}
   \left(-L_0^{-1}\operatorname{grad}F(x_k)\right)
  \label{eq:abstract-fiber-update}
\end{equation}
stays in the current sublevel set, and, with
\(q_{\mu,L}=1-\mu/L_0\),
\begin{align}
  \Delta_k&\le q_{\mu,L}^k\Delta_0,
  &
  d(x_k,x_\star)
  &\le\sqrt{\frac{2\Delta_0}{\mu}}\,q_{\mu,L}^{k/2},
  \label{eq:abstract-fiber-rates}\\
  \norm{\operatorname{grad}F(x_k)}
  &\le\sqrt{2L_0\Delta_0}\,q_{\mu,L}^{k/2}.
  \label{eq:abstract-fiber-gradient-rate}
\end{align}
At every \(x\in\mathcal C\),
\begin{equation}
  d(x,x_\star)\le\frac{\norm{\operatorname{grad}F(x)}}{\mu},
  \qquad
  0\le F(x)-F(x_\star)
  \le\frac{\norm{\operatorname{grad}F(x)}^2}{2\mu}.
  \label{eq:abstract-fiber-certificates}
\end{equation}
The constants in \eqref{eq:abstract-fiber-certificates} are sharp.
\end{theorem}

The first-return sublevel statement is important: it uses a Hessian bound on
the current sublevel itself, rather than first enlarging its radius to cover a
putative update.  For squared distance on a Hadamard manifold with
\(-\kappa\le\operatorname{sec}\le0\), Hessian comparison supplies the sharp
radius-only majorant
\begin{equation}
  L(D)=\psi(\sqrt\kappa D),
  \qquad
  \psi(s)=
  \begin{cases}
    s\coth s,&s>0,\\
    1,&s=0.
  \end{cases}
  \label{eq:hessian-comparison-factor}
\end{equation}
The general descent ingredients are classical; the theorem records the exact
form needed to connect the canonical section to the action-specific result.

\subsection{Explicit SPD residual and nonasymptotic complexity}

For fixed \(B\in\Bact\), retain \(R_0,\mathcal Q_\Sigma,f_B\), and
\(\mathcal R(\Sigma)\) from \cref{sec:action-bundle}, and define
\begin{equation}
  W_B(\Sigma)=\mathcal R_{22}(\Sigma)^\circ.
  \label{eq:solver-residual}
\end{equation}
This is the exact whitened negative gradient:
\begin{equation}
  -\operatorname{grad}f_B(\Sigma)
  =\Sigma^{1/2}W_B(\Sigma)\Sigma^{1/2},
  \qquad
  \norm{\operatorname{grad}f_B(\Sigma)}_\Sigma
  =\fro{W_B(\Sigma)}.
  \label{eq:residual-gradient-identity}
\end{equation}

Given any \(\Sigma_0\in\mathscr P_n\), put
\begin{equation}
  D_0=\dai(R_0,\mathcal Q_{\Sigma_0}),
  \qquad
  L_0=\psi(D_0/\sqrt2),
  \qquad
  q_0=1-L_0^{-1}.
  \label{eq:solver-constants}
\end{equation}
If \(D_0=0\), the initial controller is canonical.  Otherwise define
\begin{equation}
  \chi_0=-\log q_0
  =-\log(1-L_0^{-1})
  \label{eq:solver-discrete-exponent}
\end{equation}
and run
\begin{equation}
  \boxed{
  \Sigma_{k+1}
  =\Sigma_k^{1/2}
   \exp\!\left(L_0^{-1}W_B(\Sigma_k)\right)
   \Sigma_k^{1/2}.}
  \label{eq:certified-solver-update}
\end{equation}

\begin{theorem}[Global certified action solver]
\label{thm:certified-action-solver}
Let \(A\) have full column rank, let \(B\in\Bact\), and let
\(\Sigma_\star\) be the hidden coordinate of \(s_A(B)\).  From every
\(\Sigma_0\in\mathscr P_n\), \eqref{eq:certified-solver-update} preserves
determinant one and satisfies
\begin{align}
  f_B(\Sigma_{k+1})
  &\le f_B(\Sigma_k)
  -\frac{1}{2L_0}\fro{W_B(\Sigma_k)}^2,
  \label{eq:solver-descent}\\
  f_B(\Sigma_k)-f_B(\Sigma_\star)
  &\le \frac{D_0^2}{2}q_0^k,
  \label{eq:solver-value-rate}\\
  \dai(\Sigma_k,\Sigma_\star)
  &=\dai(\Phi_A(B,\Sigma_k),s_A(B))
  \le D_0q_0^{k/2},
  \label{eq:solver-distance-rate}\\
  \fro{W_B(\Sigma_k)}
  &\le \sqrt{L_0}\,D_0q_0^{k/2}.
  \label{eq:solver-residual-rate}
\end{align}
The observable residual gives the sharp posterior certificates
\begin{align}
  \dai(\Phi_A(B,\Sigma_k),s_A(B))
  &\le\fro{W_B(\Sigma_k)},
  \label{eq:solver-distance-certificate}\\
  0\le f_B(\Sigma_k)-f_B(\Sigma_\star)
  &\le\frac12\fro{W_B(\Sigma_k)}^2.
  \label{eq:solver-gap-certificate}
\end{align}
For \(D_0>0\), value, controller, and residual tolerances are guaranteed,
respectively, after
\begin{align}
  k&\ge\frac1{\chi_0}
       \pospart{\log\frac{D_0^2}{2\varepsilon_f}},
  \label{eq:solver-value-complexity}\\
  k&\ge\frac2{\chi_0}
       \pospart{\log\frac{D_0}{\varepsilon_x}},
  \label{eq:solver-distance-complexity}\\
  k&\ge\frac2{\chi_0}
       \pospart{\log\frac{\sqrt{L_0}D_0}{\varepsilon_r}}
  \label{eq:solver-residual-complexity}
\end{align}
iterations, with ceilings understood.
\end{theorem}

The constant \(L_0=\psi(D_0/\sqrt2)\) is the sharp dimension-uniform
curvature-comparison majorant for the current sublevel; the previous
radius-doubling argument is unnecessary.  The denominators above invert the
discrete factor \(q_0\) exactly; \(1/\chi_0\le L_0\) recovers the simpler
exponential estimates.

\subsection{Active realization and spectral-arithmetic cost}

The operation counts below use a spectral-arithmetic model: dense
factorizations and matrix functions on an \(r\times r\) symmetric matrix cost
\(O(r^3)\), with scalar elementary functions counted at unit cost.  These are
algebraic operation counts, not bit-complexity or floating-point error bounds.

Let
\[
  \mathcal W_A
  =\operatorname{span}(\operatorname{col}A,\operatorname{col}B),
  \quad
  r=\dim\mathcal W_A\le2m,
  \quad
  s=r-m,
  \quad
  q=d-r.
\]
The active-adapted construction in
\cref{eq:active-residual-splitting} represents the hidden state by
\(\Sigma_a\oplus\beta I_q\), with
\(\det\Sigma_a\,\beta^q=1\), and the full residual norm by
\(\fro{W_a}^2+qw_c^2\).  All complement terms are omitted when \(q=0\).

\begin{corollary}[Active canonical solver]
\label{cor:active-core-solver}
Starting from \(\Sigma_{a,0}=I_s\) and, when present, \(\beta_0=1\), the
updates
\begin{equation}
  \Sigma_{a,k+1}
  =\Sigma_{a,k}^{1/2}e^{W_{a,k}/L_0}\Sigma_{a,k}^{1/2},
  \qquad
  \beta_{k+1}=\beta_ke^{w_{c,k}/L_0}
  \label{eq:active-solver-update}
\end{equation}
are exactly the full iteration \eqref{eq:certified-solver-update}.  They
preserve \(P_kA=B\), \(P_k\succ0\), and \(\det P_k=1\), while the stopping
test \(\fro{W_{a,k}}^2+qw_{c,k}^2\le\varepsilon^2\) implies controller
error at most \(\varepsilon\).

Preprocessing costs \(O(dm^2+r^3)\) spectral-arithmetic operations and
\(O(dm+r^2)\) storage.  Each iteration costs \(O(r^3)\).  For \(D_0>0\), the total
spectral-arithmetic
cost is
\begin{equation}
  O\!\left(
  dm^2+
  r^3\left\{
  1+\frac1{\chi_0}
  \pospart{\log\frac{\sqrt{L_0}D_0}{\varepsilon}}
  \right\}\right),
  \qquad r\le2m.
  \label{eq:active-total-complexity}
\end{equation}
The implicit controller application costs \(O(dr+r^2)\).  The same formulas
cover \(q=0\); if \(s=0\), the solver has no variable and the canonical
controller is
\[
  s_A(B)=EME^\top+\rho(I-EE^\top).
\]
\end{corollary}

\subsection{Conditional inexact residuals}

Finite-precision propagation requires certified numerical information that is
separate from the geometric solver theorem.  Assume a matrix-function routine
supplies a rigorous initial radius majorant \(\widehat D_0\ge D_0\), a
symmetric trace-free direction
\(\widetilde W_k=W_B(\Sigma_k)+E_k\), and an error majorant
\(b_k\ge\fro{E_k}\).  Such bounds require a certified matrix-function
implementation; they are not produced by the analysis below
\citep{higham2008functions}.

Fix \(0\le\delta<1\) and require at every accepted step
\begin{equation}
  b_k\le\frac{\delta}{1+\delta}\fro{\widetilde W_k}.
  \label{eq:computable-relative-error-test}
\end{equation}
This observable test implies
\begin{equation}
  \fro{E_k}\le\delta\fro{W_B(\Sigma_k)}.
  \label{eq:relative-residual-error}
\end{equation}
Define
\begin{equation}
  \widehat L_0=\psi(\widehat D_0/\sqrt2),
  \quad
  c_\delta=\left(\frac{1-\delta}{1+\delta}\right)^2,
  \quad
  \widehat q_\delta=1-\frac{c_\delta}{\widehat L_0},
  \quad
  \eta=\frac{1-\delta}{\widehat L_0(1+\delta)^2}.
  \label{eq:inexact-step}
\end{equation}

\begin{corollary}[Conditional inexact global rate]
\label{cor:inexact-action-solver}
If the supplied majorants satisfy the preceding conditions, the update
\[
  \Sigma_{k+1}
  =\Sigma_k^{1/2}\exp(\eta\widetilde W_k)\Sigma_k^{1/2}
\]
has nonincreasing objective radius and
\begin{align}
  f_B(\Sigma_k)-f_B(\Sigma_\star)
  &\le\frac{\widehat D_0^2}{2}\widehat q_\delta^k,
  \label{eq:inexact-value-rate}\\
  \dai(\Phi_A(B,\Sigma_k),s_A(B))
  &\le\widehat D_0\widehat q_\delta^{k/2},
  \label{eq:inexact-distance-rate}\\
  \fro{\widetilde W_k}
  &\le\frac{(1+\delta)^2}{1-\delta}
  \sqrt{\widehat L_0}\widehat D_0\widehat q_\delta^{k/2}.
  \label{eq:inexact-residual-rate}
\end{align}
The observed residual also certifies
\begin{equation}
  \dai(\Phi_A(B,\Sigma_k),s_A(B))
  \le\frac{\fro{\widetilde W_k}}{1-\delta},
  \qquad
  f_B(\Sigma_k)-f_B(\Sigma_\star)
  \le\frac{\fro{\widetilde W_k}^2}{2(1-\delta)^2}.
  \label{eq:inexact-certificates}
\end{equation}
\end{corollary}

This corollary is an error-propagation guarantee, not a bit-complexity theorem
for computing matrix logarithms, roots, or exponentials.

\section{Canonical Multi-Secant Recovery}
\label{sec:multisecant-recovery}

The certified solver computes one canonical action projection.  Nested action
fibers turn those projections into a finite inverse-shape recovery law.  Let
\(H\in\SPD^d\) be a fixed quadratic Hessian and define its determinant-one
inverse shape
\begin{equation}
  P_\star=c_HH^{-1},
  \qquad c_H=(\det H)^{1/d}.
  \label{eq:inverse-shape}
\end{equation}
The scalar \(c_H\) is assumed known, or supplied by an independent scalar
gauge.  This information requirement is part of the theorem.

Let \(S^{(1)},\ldots,S^{(K)}\) have strictly nested column spaces, delete
dependent columns, and write \(r_k=\rank S^{(k)}\).  Define the cumulative
shape-normalized actions
\begin{equation}
  A_k=HS^{(k)},
  \qquad B_k=c_HS^{(k)},
  \qquad
  \mathscr C_k=\{P\in\Pdet:PA_k=B_k\}.
  \label{eq:nested-action-fibers}
\end{equation}
Then \(P_\star\in\mathscr C_k\) and
\(\mathscr C_{k+1}\subset\mathscr C_k\).  Starting from any
\(P_0\in\Pdet\), define
\begin{equation}
  P_k=\argminop_{P\in\mathscr C_k}\dai(P_{k-1},P).
  \label{eq:canonical-multisecant-update}
\end{equation}

\begin{theorem}[Nested canonical multi-secant identification]
\label{thm:nested-multisecant}
For the exact outer projections in \eqref{eq:canonical-multisecant-update},
every update exists uniquely and
preserves all earlier cumulative secants.  For every \(k\),
\begin{equation}
  \boxed{
  \dai(P_k,P_\star)^2+
  \dai(P_{k-1},P_k)^2
  \le\dai(P_{k-1},P_\star)^2.}
  \label{eq:multisecant-pythagorean}
\end{equation}
Consequently,
\begin{align}
  \sum_{j=1}^k\dai(P_{j-1},P_j)^2
  &\le\dai(P_0,P_\star)^2-
       \dai(P_k,P_\star)^2,
  \label{eq:multisecant-energy}\\
  \sum_{j=1}^k\dai(P_{j-1},P_j)
  &\le\sqrt{k}\,\dai(P_0,P_\star).
  \label{eq:multisecant-length}
\end{align}
The determinant-one identification threshold is sharp:
\begin{equation}
  r_K=d-1\quad\Longrightarrow\quad P_K=P_\star,
  \label{eq:sharp-identification-sufficiency}
\end{equation}
whereas any cumulative rank \(r\le d-2\) leaves infinitely many feasible
determinant-one SPD completions.  If each pre-final round contributes exactly
\(b\) new independent probe directions and the final round contributes the
remainder, exact recovery occurs after
\begin{equation}
  K=\left\lceil\frac{d-1}{b}\right\rceil
  \label{eq:block-identification-rounds}
\end{equation}
rounds.  Each exact outer projection reduces by AIRM congruence to the
identity-reference problem in \cref{thm:certified-action-solver}; the inner
iteration converges globally to it and supplies a residual certificate for a
finite approximation.
\end{theorem}

The known scalar gauge cannot be dropped.  Ordinary Hessian-vector secants in
\(d-1\) directions do not determine the inverse shape.  In dimension two,
\[
  H_t=\operatorname{diag}(1,t),\qquad s=e_1
\]
gives the same observation \(H_ts=e_1\) for every \(t>0\), while
\[
  (\det H_t)^{1/2}H_t^{-1}
  =\operatorname{diag}(\sqrt t,1/\sqrt t)
\]
varies with \(t\).  Thus \eqref{eq:sharp-identification-sufficiency} is a
sharp theorem for shape-normalized exact actions, rather than a claim of
scale-free identification from ordinary secants.  Acquisition of \(c_H\) is
external to the action oracle and is not claimed to be cheaper than an
additional probe.

\section{Discussion}
\label{sec:discussion}

\paragraph{One canonical-fiber chain.}
Variational reduction selects a hidden minimizer; its vertical Hessian governs
both response and computation.  In \cref{cor:action-interaction-response},
\(H_B^{-1}\) maps mechanism perturbations to interaction curvature and maps
visible-action perturbations to canonical-section susceptibility.  For the
unperturbed action problem, total geodesy and the logarithmic residual turn
the same canonical fiber into an explicit strongly convex optimization
problem.  This is the paper's central chain:
\[
  \text{reduction}\;\longrightarrow\;
  \text{vertical response}\;\longrightarrow\;
  \text{action bundle}\;\longrightarrow\;
  \text{posterior-certified active solver}.
\]

\paragraph{What is action specific.}
General Hadamard first-order convergence is prior art.  The action-specific
content is the global parameterization of every determinant-one \(PA=B\)
fiber, its total geodesy and analytic canonical section, the exact residual,
the current-sublevel constant attained inside the action family, and an active
implementation whose spectral operations are \(r\times r\), with a strict
dimension reduction when \(r<d\).  Nested normalized actions then convert the
bundle geometry into a sharp rank-\(d-1\) finite-identification law.

\paragraph{Numerical and modeling boundary.}
The exact solver has a spectral-arithmetic operation bound.  The inexact corollary is
conditional on rigorous matrix-function error and radius majorants; it
propagates supplied certificates but does not construct them or establish bit
complexity.  The recovery law additionally requires the independent scalar
\(c_H=(\det H)^{1/d}\).  The closed model is therefore full SPD, feasible
exact actions, and fixed determinant scale.  Structured or disconnected
fibers, noisy actions, conditioning as \(A^\top B\) approaches the SPD
boundary, unknown scale, and empirical optimizer speedups remain outside the
present claims.

\section*{Reproducibility Statement}

This is a theory-only paper.  Reproducibility consists of checking the stated
assumptions and the complete proofs in the appendices.  The manuscript uses no
datasets, model checkpoints, trained models, or empirical performance claims.
Every formal claim is supported by an analytic proof in the appendices; no
unreleased numerical check is used as evidence.  The explicit iteration is
fully specified by the action matrices and standard SPD matrix functions; the
conditional inexact theorem additionally assumes rigorous initial-radius and
matrix-function error majorants satisfying
\cref{eq:computable-relative-error-test}.

\section*{Broader Impact Statement}

The paper develops mathematical foundations for adaptive optimization and
does not release a deployable learning system or collect data.  Its likely
impact is indirect, through analysis or design of optimization methods.  The
explicit information requirements and the distinction between controlled
geometric response and causal interpretation are intended to discourage
unsupported recovery claims from passive optimizer traces.

\bibliographystyle{tmlr}
\bibliography{references}

\begin{thebibliography}{45}
\providecommand{\natexlab}[1]{#1}
\providecommand{\url}[1]{\texttt{#1}}
\expandafter\ifx\csname urlstyle\endcsname\relax
  \providecommand{\doi}[1]{doi: #1}\else
  \providecommand{\doi}{doi: \begingroup \urlstyle{rm}\Url}\fi

\bibitem[Absil et~al.(2008)Absil, Mahony, and Sepulchre]{absil2008optimization}
P.-A. Absil, Robert Mahony, and Rodolphe Sepulchre.
\newblock \emph{Optimization Algorithms on Matrix Manifolds}.
\newblock Princeton University Press, 2008.

\bibitem[Amari(1998)]{amari1998natural}
Shun-ichi Amari.
\newblock Natural gradient works efficiently in learning.
\newblock \emph{Neural Computation}, 10\penalty0 (2):\penalty0 251--276, 1998.
\newblock \doi{10.1162/089976698300017746}.

\bibitem[Ba\v{c}ak(2014)]{bacak2014convex}
Miroslav Ba\v{c}ak.
\newblock \emph{Convex Analysis and Optimization in Hadamard Spaces}.
\newblock De Gruyter, 2014.
\newblock \doi{10.1515/9783110361629}.

\bibitem[Bhatia(2007)]{bhatia2007positive}
Rajendra Bhatia.
\newblock \emph{Positive Definite Matrices}.
\newblock Princeton University Press, 2007.

\bibitem[Bhatia \& Jain(2014)Bhatia and Jain]{bhatia2014approximation}
Rajendra Bhatia and Tanvi Jain.
\newblock Approximation problems in the riemannian metric on positive definite
  matrices.
\newblock \emph{Annals of Functional Analysis}, 5\penalty0 (2):\penalty0
  118--126, 2014.
\newblock \doi{10.15352/afa/1396833507}.

\bibitem[Bonnans \& Shapiro(2000)Bonnans and Shapiro]{bonnans2000perturbation}
J.~Fr{\'e}d{\'e}ric Bonnans and Alexander Shapiro.
\newblock \emph{Perturbation Analysis of Optimization Problems}.
\newblock Springer New York, 2000.
\newblock \doi{10.1007/978-1-4612-1394-9}.

\bibitem[Boumal(2023)]{boumal2023introduction}
Nicolas Boumal.
\newblock \emph{An Introduction to Optimization on Smooth Manifolds}.
\newblock Cambridge University Press, 2023.
\newblock \doi{10.1017/9781009166164}.

\bibitem[Boutet et~al.(2022)Boutet, Degroote, and
  Haelterman]{boutet2022grouped}
Nicolas Boutet, Joris Degroote, and Rob Haelterman.
\newblock A symmetric grouped and ordered multi-secant quasi-newton update
  formula.
\newblock \emph{Optimization Methods and Software}, 37\penalty0 (6):\penalty0
  1965--1986, 2022.
\newblock \doi{10.1080/10556788.2022.2053970}.

\bibitem[Calamai \& Mor{\'e}(1987)Calamai and Mor{\'e}]{calamai1987bounds}
Paul~H. Calamai and Jorge~J. Mor{\'e}.
\newblock Quasi-newton updates with bounds.
\newblock \emph{SIAM Journal on Numerical Analysis}, 24\penalty0 (6):\penalty0
  1434--1441, 1987.
\newblock \doi{10.1137/0724092}.

\bibitem[Crawshaw et~al.(2025)Crawshaw, Modi, Liu, and
  Gower]{crawshaw2025muonvariants}
Michael Crawshaw, Chirag Modi, Mingrui Liu, and Robert~M. Gower.
\newblock An exploration of non-euclidean gradient descent: Muon and its many
  variants, 2025.
\newblock URL \url{https://arxiv.org/abs/2510.09827}.

\bibitem[Dai \& Yamashita(2014)Dai and Yamashita]{dai2014sparse}
Yu-Hong Dai and Nobuo Yamashita.
\newblock Analysis of sparse quasi-newton updates with positive definite matrix
  completion.
\newblock \emph{Journal of the Operations Research Society of China},
  2\penalty0 (1):\penalty0 39--56, 2014.
\newblock \doi{10.1007/S40305-014-0039-X}.

\bibitem[Dennis \& Schnabel(1979)Dennis and Schnabel]{dennis1979leastchange}
John~E. Dennis, Jr. and Robert~B. Schnabel.
\newblock Least change secant updates for quasi-newton methods.
\newblock \emph{SIAM Review}, 21\penalty0 (4):\penalty0 443--459, 1979.
\newblock \doi{10.1137/1021091}.

\bibitem[Do{\u{g}}an et~al.(2025)Do{\u{g}}an, Erg{\"u}r, and
  Tsigaridas]{dogan2025geodesicpreconditioning}
M.~Levent Do{\u{g}}an, Alperen Erg{\"u}r, and Elias Tsigaridas.
\newblock Optimal preconditioning is a geodesically convex optimization
  problem, 2025.
\newblock URL \url{https://arxiv.org/abs/2512.06618}.

\bibitem[Duchi et~al.(2011)Duchi, Hazan, and Singer]{duchi2011adagrad}
John Duchi, Elad Hazan, and Yoram Singer.
\newblock Adaptive subgradient methods for online learning and stochastic
  optimization.
\newblock \emph{Journal of Machine Learning Research}, 12:\penalty0 2121--2159,
  2011.
\newblock URL \url{https://jmlr.org/papers/v12/duchi11a.html}.

\bibitem[{Essential AI} et~al.(2025){Essential AI}, Shah, Polloreno, Stratos,
  Monk, Chaluvaraju, Hojel, Ma, Thomas, Tanwer, Shah, Nguyen, Smith, Callahan,
  Pust, Parmar, Rushton, Mazarakis, Kapila, Srivastava, Singla, Romanski,
  Vanjani, and Vaswani]{essentialai2025muon}
{Essential AI}, Ishaan Shah, Anthony~M. Polloreno, Karl Stratos, Philip Monk,
  Adarsh Chaluvaraju, Andrew Hojel, Andrew Ma, Anil Thomas, Ashish Tanwer,
  Darsh~J. Shah, Khoi Nguyen, Kurt Smith, Michael Callahan, Michael Pust, Mohit
  Parmar, Peter Rushton, Platon Mazarakis, Ritvik Kapila, Saurabh Srivastava,
  Somanshu Singla, Tim Romanski, Yash Vanjani, and Ashish Vaswani.
\newblock Practical efficiency of muon for pretraining, 2025.
\newblock URL \url{https://arxiv.org/abs/2505.02222}.

\bibitem[Fletcher(1991)]{fletcher1991variational}
Roger Fletcher.
\newblock A new variational result for quasi-newton formulae.
\newblock \emph{SIAM Journal on Optimization}, 1\penalty0 (1):\penalty0 18--21,
  1991.
\newblock \doi{10.1137/0801002}.

\bibitem[Gao \& Goldfarb(2018)Gao and Goldfarb]{gao2018block}
Wenbo Gao and Donald Goldfarb.
\newblock Block {BFGS} methods.
\newblock \emph{SIAM Journal on Optimization}, 28\penalty0 (2):\penalty0
  1205--1231, 2018.
\newblock \doi{10.1137/16M1092106}.

\bibitem[Gao et~al.(2026)Gao, Qu, Udell, and Ye]{gao2023scalable}
Wenzhi Gao, Zhaonan Qu, Madeleine Udell, and Yinyu Ye.
\newblock Scalable approximate optimal diagonal preconditioning.
\newblock \emph{Computational Optimization and Applications}, 94\penalty0
  (2):\penalty0 439--473, 2026.
\newblock \doi{10.1007/s10589-026-00770-8}.
\newblock URL \url{https://doi.org/10.1007/s10589-026-00770-8}.

\bibitem[Golub \& Pereyra(1973)Golub and Pereyra]{golub1973variableprojection}
Gene~H. Golub and Victor Pereyra.
\newblock The differentiation of pseudo-inverses and nonlinear least squares
  problems whose variables separate.
\newblock \emph{SIAM Journal on Numerical Analysis}, 10\penalty0 (2):\penalty0
  413--432, 1973.
\newblock \doi{10.1137/0710036}.

\bibitem[Gratton et~al.(2015)Gratton, Malmedy, and Toint]{gratton2015weighted}
Serge Gratton, V.~Malmedy, and Philippe~L. Toint.
\newblock Quasi-newton updates with weighted secant equations.
\newblock \emph{Optimization Methods and Software}, 30\penalty0 (4):\penalty0
  748--755, 2015.
\newblock \doi{10.1080/10556788.2014.971025}.

\bibitem[Grone et~al.(1984)Grone, Johnson, S{\'a}, and
  Wolkowicz]{grone1984positive}
Robert Grone, Charles~R. Johnson, Eduardo~M. S{\'a}, and Henry Wolkowicz.
\newblock Positive definite completions of partial hermitian matrices.
\newblock \emph{Linear Algebra and its Applications}, 58:\penalty0 109--124,
  1984.
\newblock \doi{10.1016/0024-3795(84)90207-6}.

\bibitem[Gupta et~al.(2018)Gupta, Koren, and Singer]{gupta2018shampoo}
Vineet Gupta, Tomer Koren, and Yoram Singer.
\newblock Shampoo: Preconditioned stochastic tensor optimization.
\newblock In \emph{Proceedings of the 35th International Conference on Machine
  Learning}, volume~80 of \emph{Proceedings of Machine Learning Research}, pp.\
   1842--1850. PMLR, 2018.
\newblock URL \url{https://proceedings.mlr.press/v80/gupta18a.html}.

\bibitem[Higham(2008)]{higham2008functions}
Nicholas~J. Higham.
\newblock \emph{Functions of Matrices: Theory and Computation}.
\newblock Society for Industrial and Applied Mathematics, 2008.
\newblock \doi{10.1137/1.9780898717778}.

\bibitem[Jordan(2024)]{jordan2024muon}
Keller Jordan.
\newblock Muon: An optimizer for hidden layers in neural networks.
\newblock \url{https://kellerjordan.github.io/posts/muon/}, 2024.
\newblock Blog post.

\bibitem[Kanamori \& Ohara(2013)Kanamori and Ohara]{kanamori2013bregman}
Takafumi Kanamori and Atsumi Ohara.
\newblock A bregman extension of quasi-newton updates i: An information
  geometrical framework.
\newblock \emph{Optimization Methods and Software}, 28\penalty0 (1):\penalty0
  96--123, 2013.
\newblock \doi{10.1080/10556788.2011.613073}.

\bibitem[Kingma \& Ba(2015)Kingma and Ba]{kingma2014adam}
Diederik~P. Kingma and Jimmy Ba.
\newblock Adam: A method for stochastic optimization.
\newblock In \emph{International Conference on Learning Representations}, 2015.
\newblock URL \url{https://arxiv.org/abs/1412.6980}.
\newblock arXiv:1412.6980.

\bibitem[Li(2026)]{li2026informationgeometry}
Zavier Li.
\newblock Information-induced training geometry: Exact reduction, canonical
  completion, and structured expressivity, 2026.
\newblock arXiv preprint.

\bibitem[Lim(2004)]{lim2004best}
Yongdo Lim.
\newblock Best approximation in riemannian geodesic submanifolds of positive
  definite matrices.
\newblock \emph{Canadian Journal of Mathematics}, 56\penalty0 (4):\penalty0
  776--793, 2004.
\newblock \doi{10.4153/CJM-2004-035-5}.

\bibitem[Lu \& Pong(2011)Lu and Pong]{lu2011minimizing}
Zhaosong Lu and Ting~Kei Pong.
\newblock Minimizing condition number via convex programming.
\newblock \emph{SIAM Journal on Matrix Analysis and Applications}, 32\penalty0
  (4):\penalty0 1193--1211, 2011.
\newblock \doi{10.1137/100795097}.

\bibitem[Mar{\'e}chal \& Ye(2009)Mar{\'e}chal and Ye]{marechal2009optimizing}
Pierre Mar{\'e}chal and Jane~J. Ye.
\newblock Optimizing condition numbers.
\newblock \emph{SIAM Journal on Optimization}, 20\penalty0 (2):\penalty0
  935--947, 2009.
\newblock \doi{10.1137/080740544}.

\bibitem[Martens \& Grosse(2015)Martens and Grosse]{martens2015kfac}
James Martens and Roger Grosse.
\newblock Optimizing neural networks with {Kronecker}-factored approximate
  curvature.
\newblock In \emph{Proceedings of the 32nd International Conference on Machine
  Learning}, volume~37 of \emph{Proceedings of Machine Learning Research}, pp.\
   2408--2417. PMLR, 2015.
\newblock URL \url{https://proceedings.mlr.press/v37/martens15.html}.

\bibitem[Morwani et~al.(2024)Morwani, Shapira, Vyas, Malach, Kakade, and
  Janson]{morwani2024shampoo}
Depen Morwani, Itai Shapira, Nikhil Vyas, Eran Malach, Sham~M. Kakade, and
  Lucas Janson.
\newblock A new perspective on {Shampoo}'s preconditioner, 2024.
\newblock URL \url{https://arxiv.org/abs/2406.17748}.

\bibitem[Passy \& Prisman(1984)Passy and Prisman]{passy1984secant}
U.~Passy and E.~Z. Prisman.
\newblock Secant relations versus positive definiteness in quasi-newton
  methods.
\newblock \emph{Journal of Optimization Theory and Applications}, 44\penalty0
  (4):\penalty0 681--687, 1984.
\newblock \doi{10.1007/BF00938401}.

\bibitem[Pennec et~al.(2006)Pennec, Fillard, and Ayache]{pennec2006riemannian}
Xavier Pennec, Pierre Fillard, and Nicholas Ayache.
\newblock A riemannian framework for tensor computing.
\newblock \emph{International Journal of Computer Vision}, 66\penalty0
  (1):\penalty0 41--66, 2006.
\newblock \doi{10.1007/s11263-005-3222-z}.

\bibitem[Qu et~al.(2025)Qu, Gao, Hinder, Ye, and Zhou]{qu2024optimal}
Zhaonan Qu, Wenzhi Gao, Oliver Hinder, Yinyu Ye, and Zhengyuan Zhou.
\newblock Optimal diagonal preconditioning.
\newblock \emph{Operations Research}, 73\penalty0 (3):\penalty0 1479--1495,
  2025.
\newblock \doi{10.1287/opre.2022.0592}.

\bibitem[Rockafellar(1970)]{rockafellar1970convex}
R.~Tyrrell Rockafellar.
\newblock \emph{Convex Analysis}.
\newblock Princeton University Press, 1970.

\bibitem[Rockafellar \& Wets(1998)Rockafellar and
  Wets]{rockafellar1998variational}
R.~Tyrrell Rockafellar and Roger J.-B. Wets.
\newblock \emph{Variational Analysis}, volume 317 of \emph{Grundlehren der
  mathematischen Wissenschaften}.
\newblock Springer Berlin Heidelberg, 1998.
\newblock \doi{10.1007/978-3-642-02431-3}.

\bibitem[Schnabel(1983)]{schnabel1983multiple}
Robert~B. Schnabel.
\newblock Quasi-newton methods using multiple secant equations.
\newblock Technical report, Defense Technical Information Center, 1983.

\bibitem[Sra \& Hosseini(2015)Sra and Hosseini]{sra2015conic}
Suvrit Sra and Reshad Hosseini.
\newblock Conic geometric optimization on the manifold of positive definite
  matrices.
\newblock \emph{SIAM Journal on Optimization}, 25\penalty0 (1):\penalty0
  713--739, 2015.
\newblock \doi{10.1137/140978168}.

\bibitem[Tanaka \& Nakata(2014)Tanaka and Nakata]{tanaka2014condition}
Mirai Tanaka and Kazuhide Nakata.
\newblock Positive definite matrix approximation with condition number
  constraint.
\newblock \emph{Optimization Letters}, 8\penalty0 (3):\penalty0 939--947, 2014.
\newblock \doi{10.1007/s11590-013-0632-7}.

\bibitem[Tian(2013)]{tian2013hermitian}
Yongge Tian.
\newblock Equalities and inequalities for hermitian solutions and hermitian
  definite solutions of the two matrix equations {$AX=B$} and {$AXA^*=B$}.
\newblock \emph{Aequationes Mathematicae}, 86\penalty0 (1--2):\penalty0
  107--135, 2013.
\newblock \doi{10.1007/s00010-012-0179-1}.

\bibitem[Tumpach \& Larotonda(2024)Tumpach and Larotonda]{tumpach2024totally}
Alice~Barbara Tumpach and Gabriel Larotonda.
\newblock Totally geodesic submanifolds in the manifold {SPD} of symmetric
  positive-definite real matrices.
\newblock \emph{Information Geometry}, 7\penalty0 (S2):\penalty0 913--942,
  2024.
\newblock \doi{10.1007/s41884-024-00146-z}.

\bibitem[Vyas et~al.(2024)Vyas, Morwani, Zhao, Kwun, Shapira, Brandfonbrener,
  Janson, and Kakade]{vyas2025soap}
Nikhil Vyas, Depen Morwani, Rosie Zhao, Mujin Kwun, Itai Shapira, David
  Brandfonbrener, Lucas Janson, and Sham Kakade.
\newblock {SOAP}: Improving and stabilizing {Shampoo} using {Adam}, 2024.
\newblock URL \url{https://arxiv.org/abs/2409.11321}.

\bibitem[Yamashita(2007)]{yamashita2007sparse}
Nobuo Yamashita.
\newblock Sparse quasi-newton updates with positive definite matrix completion.
\newblock \emph{Mathematical Programming}, 115\penalty0 (1):\penalty0 1--30,
  2007.
\newblock \doi{10.1007/s10107-007-0137-1}.

\bibitem[Zhang \& Sra(2016)Zhang and Sra]{zhang2016firstorder}
Hongyi Zhang and Suvrit Sra.
\newblock First-order methods for geodesically convex optimization.
\newblock In \emph{29th Annual Conference on Learning Theory}, volume~49 of
  \emph{Proceedings of Machine Learning Research}, pp.\  1617--1638. PMLR,
  2016.
\newblock URL \url{https://proceedings.mlr.press/v49/zhang16b.html}.

\end{thebibliography}

\appendix
\section{Proofs for Variational Reduction and Response}
\label{app:reduction-response}

\subsection{Algebraic reduction}

\begin{proof}[Proof of \cref{thm:algebraic-reduction}]
Fix \(z\in\Zspace\).  Partitioning the composite fiber by its intermediate
visible value gives
\[
\begin{aligned}
  ((\rho\circ\pi)_!E)(z)
  &=\inf_{x:\rho(\pi(x))=z}E(x)\\
  &=\inf_{y:\rho(y)=z}\inf_{x:\pi(x)=y}E(x)\\
  &=(\rho_!(\pi_!E))(z).
\end{aligned}
\]
Likewise,
\[
\begin{aligned}
  \inf_x\{E(x)+\ell(\pi(x))\}
  &=\inf_y\inf_{\pi(x)=y}\{E(x)+\ell(y)\}\\
  &=\inf_y\{(\pi_!E)(y)+\ell(y)\}.
\end{aligned}
\]
The base-potential law follows because \(\varphi(\pi(x))=\varphi(y)\) is
constant on each fiber.  The constraint law follows from the definition of
the indicator energy.  No attainment assumption is used.
\end{proof}

\subsection{Existence and uniqueness of the canonical lift}

\begin{proof}[Proof of \cref{thm:canonical-lift}]
Fix \(y\) with finite reduced energy.  Coercivity places every minimizing
sequence in a bounded sublevel set.  Properness makes its closure compact, and
lower semicontinuity gives a minimizer in the closed fiber.

If \(x_0\ne x_1\) were two minimizers, their fiber geodesic would satisfy
\[
  E(x_0\#_{1/2}x_1)
  \le\frac12E(x_0)+\frac12E(x_1)
  -\frac\mu8d(x_0,x_1)^2,
\]
contradicting minimality.  Thus the minimizer is unique.

For equivariance, \(g s_E(y)\) lies in the fiber over \(gy\) and has the same
energy as \(s_E(y)\).  Uniqueness on that fiber gives
\(s_E(gy)=g s_E(y)\).
\end{proof}

\subsection{Smooth minimizer response}

\begin{proof}[Proof of \cref{thm:response-schur}]
We first show continuity of the global minimizer section.  Let \(z_k\to z\).
Choose one fixed trial point near \(\sigma_\star(z)\).  Continuity of \(E\)
uniformly bounds the optimal values for all sufficiently large \(k\).  Local
uniform inf-compactness therefore places \(\sigma_\star(z_k)\) in a common
relatively compact set.  Every convergent subsequence has a limit
\(\bar\sigma\).  Optimality and continuity give, for every trial \(\sigma\),
\[
  E(z,\bar\sigma)
  =\lim E(z_k,\sigma_\star(z_k))
  \le\lim E(z_k,\sigma)=E(z,\sigma).
\]
The minimizer at \(z\) is unique, so \(\bar\sigma=\sigma_\star(z)\).  Every
subsequence has the same limit, proving continuity.

The first-order condition is
\begin{equation}
  E_\sigma(z,\sigma_\star(z))=0.
  \label{eq:appendix-stationarity}
\end{equation}
Its derivative with respect to \(\sigma\) is the invertible matrix \(H\).
The implicit-function theorem gives a local \(C^r\) critical section.  The
continuous global minimizer section agrees with it near every point and is
therefore \(C^r\).  Differentiating
\cref{eq:appendix-stationarity} gives
\[
  E_{\sigma z}+H D_z\sigma_\star=0,
\]
which proves \cref{eq:hidden-response}.  The chain rule gives
\[
  D_z\bar E
  =E_z+E_\sigma D_z\sigma_\star=E_z,
\]
and a second derivative gives
\[
  D^2_{zz}\bar E
  =E_{zz}+E_{z\sigma}D_z\sigma_\star
  =E_{zz}-E_{z\sigma}H^{-1}E_{\sigma z}.
\]

For nested variables \(\sigma=(\alpha,\beta)\), let \(q\) be the quadratic
form defined by the complete Hessian at the optimum.  Positive definiteness of
the joint hidden block makes both quadratic minimizations unique.  Algebraic
reduction gives
\[
  \inf_{\delta\alpha,\delta\beta}
  q(\delta z,\delta\alpha,\delta\beta)
  =\inf_{\delta\alpha}
  \left(\inf_{\delta\beta}
  q(\delta z,\delta\alpha,\delta\beta)\right).
\]
The matrix of a partially minimized quadratic form is the corresponding Schur
complement.  Equality for every \(\delta z\) proves
\cref{eq:schur-associativity}.
\end{proof}

\subsection{Interaction curvature and finite contrasts}

\begin{proof}[Proof of \cref{thm:interaction-curvature}]
For \cref{eq:affine-interventions},
\[
  E_{uu}=0,
  \qquad
  E_{\sigma u}=G.
\]
The \(u\)-block of \cref{eq:schur-effective-hessian} is therefore
\[
  D^2_{uu}\bar E=-G^*H^{-1}G.
\]
For every \(a\in\mathbb R^p\),
\[
  \ip{a}{D^2_{uu}\bar E\,a}
  =-\ip{Ga}{H^{-1}Ga}\le0,
\]
because \(H^{-1}\succ0\).  Expanding in coordinate vectors gives
\cref{eq:pair-interaction-curvature}.  A mixed entry is zero exactly when the
two corresponding covectors are orthogonal under the inner product induced by
\(H^{-1}\).  Finally, without differentiability,
\[
  \bar E(y,u)=\inf_\sigma
  \left(E_0(y,\sigma)+\sum_i u_i\Delta_i(y,\sigma)\right)
\]
is a pointwise infimum of affine functions of \(u\), hence concave.
\end{proof}

\subsection{Finite contrasts and order of reduction}
\label{app:finite-contrasts}

\begin{corollary}[Factorial finite differences]
\label{cor:factorial-contrast}
Assume \([0,1]^p\subset\Uspace\).  For \(T\subseteq[p]\), define
\begin{equation}
  \delta_T\bar E
  =\sum_{B\subseteq T}(-1)^{|T|-|B|}
  \bar E(y,\mathbf1_B).
  \label{eq:factorial-contrast}
\end{equation}
If \(\bar E\in C^{|T|}\), then
\begin{equation}
  \delta_T\bar E
  =\int_{[0,1]^T}
  \partial_T^{|T|}\bar E(u_T,0_{T^c})\,\dd u_T.
  \label{eq:factorial-integral}
\end{equation}
\end{corollary}

\begin{proof}
For a single coordinate, the fundamental theorem of calculus gives
\[
  \bar E(e_i)-\bar E(0)
  =\int_0^1\partial_{u_i}\bar E(t e_i)\,\dd t.
\]
Apply the same identity successively to each coordinate in \(T\).  The product
of difference operators expands as
\[
  \prod_{i\in T}(\mathsf T_i-I)\bar E(0)
  =\sum_{B\subseteq T}(-1)^{|T|-|B|}\bar E(\mathbf1_B),
\]
and repeated integration gives
\[
  \prod_{i\in T}(\mathsf T_i-I)\bar E(0)
  =\int_{[0,1]^T}\partial_T^{|T|}\bar E(u_T,0_{T^c})\,\dd u_T.
\]
For \(T=\{i,j\}\), substitute
\cref{eq:pair-interaction-curvature} to obtain
\cref{eq:pair-factorial-integral}.
\end{proof}

Reduction need not commute with finite differencing.  For
\[
  E_0(\sigma)=\sigma^2,\qquad E_1(\sigma)=(\sigma-1)^2,
\]
both reduced minima are zero, but
\(\inf_\sigma(E_1-E_0)=\inf_\sigma(1-2\sigma)=-\infty\).
Thus hidden states must be reduced first.  If every perturbation
\(\Delta_i\) is constant on each fiber, the minimizer is unchanged and all
higher-order contrasts vanish; this is the exact commuting case.

\section{Proof of the Global SPD Action Bundle}
\label{app:action-bundle}

\begin{proof}[Proof of \cref{thm:action-bundle}]
We proceed in seven steps.

\paragraph{Step 1: coordinates on the action base.}
Put \(Y=BC^{-1}\).  Since \(Q=[U,V]\) is orthogonal,
\[
  Y=UM+VN,
  \qquad M=U^\top Y,
  \qquad N=V^\top Y.
\]
Moreover,
\[
  A^\top B=C^\top M C,
\]
so \(A^\top B\in\SPD^m\) exactly when \(M\in\SPD^m\).  The map
\(B\mapsto(M,N)\) has the analytic inverse
\begin{equation}
  B=(UM+VN)C.
  \label{eq:appendix-base-inverse}
\end{equation}
Thus \(\Bact\cong\SPD^m\times\mathbb R^{n\times m}\).

\paragraph{Step 2: every SPD completion.}
Suppose \(P\in\SPD^d\) satisfies \(PA=B\).  Because \(A=UC\), this is
equivalent to \(PU=Y\).  In the \(Q\) coordinates, symmetry fixes the first
block row and column:
\[
  \widehat P:=Q^\top P Q
  =\begin{pmatrix}M&N^\top\\N&D\end{pmatrix}.
\]
The Schur-complement criterion gives
\[
  \widehat P\succ0
  \quad\Longleftrightarrow\quad
  S:=D-NM^{-1}N^\top\succ0.
\]
With \(T=NM^{-1}\), every positive solution and only such a solution is
\begin{equation}
  \widehat P
  =L\begin{pmatrix}M&0\\0&S\end{pmatrix}L^\top.
  \label{eq:appendix-all-completions}
\end{equation}
Since \(\det L=1\),
\[
  \det P=\det M\det S.
\]
The determinant-one constraint is therefore equivalent to
\[
  S=\rho\Sigma,
  \qquad \rho=(\det M)^{-1/n},
  \qquad \det\Sigma=1,
\]
which is exactly \cref{eq:action-bundle-map}.  Conversely, direct block
multiplication and \cref{eq:appendix-base-inverse} give
\[
  \Phi_A(B,\Sigma)A
  =Q\begin{pmatrix}M\\N\end{pmatrix}C=B.
\]

\paragraph{Step 3: analytic inverse.}
Given \(P\in\Pdet\), recover \(B=PA\), then \(M(B),N(B)\), and the unique
Schur complement
\[
  S=D-NM^{-1}N^\top.
\]
The unique determinant-one hidden coordinate is
\begin{equation}
  \Sigma=(\det M)^{1/n}S.
  \label{eq:appendix-hidden-inverse}
\end{equation}
All operations are analytic on their SPD domains.  Thus \(\Phi_A\) is a
global analytic diffeomorphism and \(\pi_A\) is projection onto its first
factor.

\paragraph{Step 4: fiber geometry.}
Fix \(B\).  Before congruence by \(QL\), the fiber is
\[
  \mathcal C_B
  =\{\operatorname{diag}(M,\rho\Sigma):\Sigma\in\mathscr P_n\}.
\]
The ambient AIRM geodesic between two such points is
\[
  \operatorname{diag}(M,\rho\Sigma_0)
  \#_t
  \operatorname{diag}(M,\rho\Sigma_1)
  =\operatorname{diag}(M,\rho(\Sigma_0\#_t\Sigma_1)).
\]
Because log determinant is affine along AIRM geodesics,
\(\det(\Sigma_0\#_t\Sigma_1)=1\).  Hence the fiber is totally geodesic.
It is closed, and block distance splitting gives
\[
  \dai(\operatorname{diag}(M,\rho\Sigma_0),
       \operatorname{diag}(M,\rho\Sigma_1))
  =\dai(\Sigma_0,\Sigma_1).
\]
Congruence by \(QL\) is an AIRM isometry, proving all fiber claims.

\paragraph{Step 5: canonical projection and analyticity.}
The determinant-one SPD slice is a finite-dimensional Hadamard manifold.  A
nonempty closed geodesically convex set has a unique metric projection, so
\cref{eq:canonical-controller} defines a global section.

In the global coordinates define
\[
  F(B,\Sigma)=\frac12\fro{\log\Phi_A(B,\Sigma)}^2.
\]
This function is analytic.  For fixed \(B\), its hidden variable parameterizes
an isometric totally geodesic fiber.  Squared distance to \(I\), restricted to
that fiber, is strongly geodesically convex.  At the unique critical point the
vertical Hessian is positive definite.  The analytic implicit-function theorem
therefore gives a local analytic minimizer \(\Sigma_\star(B)\); global
uniqueness makes these local sections agree.  Hence
\[
  s_A(B)=\Phi_A(B,\Sigma_\star(B))
\]
and \(\mathfrak C_A(B)^2=2F(B,\Sigma_\star(B))\) are analytic.

\paragraph{Step 6: equivariance.}
For \(D\in\operatorname{GL}_m\),
\[
  P(AD)=BD\quad\Longleftrightarrow\quad PA=B,
\]
so the two action fibers are identical and have the same unique projection.
For \(R\in O(d)\),
\[
  PA=B
  \quad\Longleftrightarrow\quad
  (RPR^\top)(RA)=RB.
\]
Orthogonal congruence preserves determinant and distance to \(I\), so
uniqueness gives \cref{eq:controller-equivariance}.

\paragraph{Step 7: normal equation and certificate.}
Congruence invariance reduces the hidden objective to
\[
  f_B(\Sigma)=\frac12\dai(R_0,\mathcal Q_\Sigma)^2.
\]
We use the convention that a function is \(1\)-strongly geodesically convex
when every constant-speed geodesic \(\gamma_t\) satisfies
\begin{equation}
  f(\gamma_t)
  \le (1-t)f(\gamma_0)+t f(\gamma_1)
  -\frac12t(1-t)d(\gamma_0,\gamma_1)^2.
  \label{eq:cat0-strong-convexity-convention}
\end{equation}
On a CAT(0) space this inequality holds for
\(f(\cdot)=\tfrac12d(\cdot,R_0)^2\).  The AIRM SPD manifold is Hadamard, and
restriction to the isometric totally geodesic determinant-one fiber preserves
the same constant \citep{bhatia2007positive,lim2004best}.
A whitened tangent to the determinant-one hidden block has the form
\[
  Z=\begin{pmatrix}0&0\\0&Z_{22}\end{pmatrix},
  \qquad Z_{22}=Z_{22}^\top,
  \qquad \tr Z_{22}=0.
\]
The whitened negative gradient is the tangent projection of
\(\mathcal R(\Sigma)\) in \cref{eq:action-residual}.  Orthogonality to every
trace-free \(Z_{22}\) is exactly
\(\mathcal R_{22}(\Sigma)^\circ=0\).  Strong convexity makes this condition
necessary and sufficient.

Let \(D=\dai(\Sigma,\Sigma_\star)\), let
\(\xi=\operatorname{Log}_\Sigma(\Sigma_\star)\), and let \(V(\Sigma)\) be the restricted
negative gradient.  The first-order form of
\cref{eq:cat0-strong-convexity-convention} gives
\[
  f_B(\Sigma)-f_B(\Sigma_\star)
  \le \langle V(\Sigma),\xi\rangle-\frac12D^2
  \le\norm{V(\Sigma)}D-\frac12D^2,
\]
while the same inequality based at the minimizer, whose gradient vanishes,
gives
\[
  f_B(\Sigma)-f_B(\Sigma_\star)\ge\frac12D^2.
\]
Adding the two bounds yields
\[
  D^2\le\norm{V(\Sigma)}D,
\]
and hence, also for \(D=0\),
\[
  D\le\norm{V(\Sigma)}
  =\fro{\mathcal R_{22}(\Sigma)^\circ}.
\]
Finally,
\[
  f_B(\Sigma)-f_B(\Sigma_\star)
  \le \sup_{D\ge0}
  \left(\norm{V(\Sigma)}D-\frac12D^2\right)
  =\frac12\norm{V(\Sigma)}^2.
\]
This proves the exact constants in
\cref{eq:action-distance-certificate,eq:action-value-certificate}.
\end{proof}

\subsection{Proof of active reduction}

Let \(J\) be the orthogonal reflection that equals \(+I\) on \(W\) and
\(-I\) on \(W^\perp\).  Since the columns of \(A\) and \(B\) lie in \(W\),
\[
  JA=A,
  \qquad JB=B.
\]
If \(P\in\Pfiber(B)\), then \((JPJ)A=B\), and \(JPJ\) has the same
determinant and distance from \(I\) as \(P\).  Uniqueness of the canonical
controller yields
\[
  Js_A(B)J=s_A(B).
\]
Hence the cross blocks between \(W\) and \(W^\perp\) vanish:
\(s_A(B)=H_\star\oplus C_\star\).

For fixed feasible \(H\), the inactive log-eigenvalues \(z_i\) satisfy
\[
  \sum_{i=1}^qz_i=-\log\det H.
\]
Cauchy--Schwarz gives
\[
  \fro{\log C}^2=\sum_i z_i^2
  \ge\frac{(\log\det H)^2}{q},
\]
with equality exactly when all \(z_i\) are equal.  Thus
\[
  C_\star=(\det H_\star)^{-1/q}I_q,
\]
and substitution yields \cref{eq:active-controller-problem}.  Uniqueness
follows from uniqueness of \(s_A(B)\).

\section{Proofs for the Certified Solver and Recovery Laws}
\label{app:solver-recovery-proofs}

This appendix proves the algorithmic and recovery results in
\cref{sec:certified-solver,sec:multisecant-recovery}.
All gradients, exponential maps, and norms below
use the affine-invariant metric.

\subsection{Curvature and squared-distance comparison}

\begin{lemma}[AIRM curvature bound]
\label{lem:airm-curvature-bound}
Every sectional curvature of \(\SPD^d\), its determinant-one slice
\(\Pdet\), and every action fiber \(\Pfiber(B)\) lies in
\begin{equation}
  -\frac12\le \operatorname{sec}\le0.
  \label{eq:airm-curvature-bound}
\end{equation}
\end{lemma}

\begin{proof}
Congruence invariance reduces the calculation to the identity, where tangent
vectors are symmetric matrices and the metric is Frobenius.  The AIRM
curvature tensor is
\[
  R(X,Y)Z=-\frac14[[X,Y],Z],
\]
and hence
\[
  \operatorname{sec}(X,Y)
  =-\frac14\frac{\fro{[X,Y]}^2}
  {\fro X^2\fro Y^2-\ip X Y_F^2}\le0.
\]
To obtain the lower bound, replace \(Y\) by its component
\(\widetilde Y\) Frobenius-orthogonal to \(X\); this leaves the commutator
unchanged and turns the denominator into
\(\fro X^2\fro{\widetilde Y}^2\).  Orthogonally diagonalize
\(X=\operatorname{diag}(\lambda_i)\).  For symmetric
\(\widetilde Y=(y_{ij})\),
\[
\begin{aligned}
  \fro{[X,\widetilde Y]}^2
  &=2\sum_{i<j}(\lambda_i-\lambda_j)^2y_{ij}^2\\
  &\le4\fro X^2\sum_{i<j}y_{ij}^2
  \le2\fro X^2\fro{\widetilde Y}^2.
\end{aligned}
\]
Substitution gives \(\operatorname{sec}\ge-1/2\).  The determinant-one
slice and the action fibers are totally geodesic, so their intrinsic
curvatures are restrictions of the ambient curvature tensor.
\end{proof}

\begin{lemma}[Squared-distance Hessian comparison]
\label{lem:squared-distance-comparison}
Let \((\mathcal M,g)\) be a finite-dimensional Hadamard manifold with
\(-\kappa\le\operatorname{sec}\le0\).  For fixed \(z\), let
\(h(x)=\tfrac12d(z,x)^2\).  Then
\begin{equation}
  g_x\preceq\operatorname{Hess}h(x)
  \preceq\psi(\sqrt\kappa R)g_x
  \quad\text{whenever }d(z,x)\le R,
  \label{eq:squared-distance-comparison}
\end{equation}
where \(\psi(s)=s\coth s\) and \(\psi(0)=1\).  In particular, \(h\) is
globally \(1\)-strongly geodesically convex.
\end{lemma}

\begin{proof}
For \(x\ne z\), put \(r=d(z,x)\), let \(T\) be the terminal unit tangent
of the geodesic from \(z\) to \(x\), and decompose
\(v=aT+v_\perp\).  The radial Hessian formula gives
\[
  \operatorname{Hess}h(v,v)
  =a^2+r\operatorname{Hess}r(v_\perp,v_\perp).
\]
Let \(J\) be the Jacobi field with \(J(0)=0\) and
\(J(r)=v_\perp\).  Nonpositive curvature and Cauchy--Schwarz applied after
parallel transport yield
\[
  \operatorname{Hess}r(v_\perp,v_\perp)=I(J,J)
  \ge\int_0^r\norm{D_tJ}^2\,\dd t
  \ge r^{-1}\norm{v_\perp}^2.
\]
For the reverse bound, compare \(J\) with
\[
  V(t)=\frac{\sinh(\sqrt\kappa t)}
  {\sinh(\sqrt\kappa r)}P_tv_\perp,
\]
using \(V(t)=(t/r)P_tv_\perp\) when \(\kappa=0\).  The index lemma and the
curvature lower bound give
\[
  I(J,J)\le I(V,V)
  \le\sqrt\kappa\coth(\sqrt\kappa r)\norm{v_\perp}^2.
\]
Therefore
\[
  \norm v^2\le\operatorname{Hess}h(v,v)
  \le\psi(\sqrt\kappa r)\norm v^2.
\]
The formula extends continuously to \(x=z\), where
\(\operatorname{Hess}h=g_z\).  Monotonicity of \(\psi\) proves the stated
radius bound, and integration of the lower Hessian bound along geodesics
proves strong convexity.
\end{proof}

\begin{proof}[Proof of \cref{thm:canonical-fiber-engine}]
Coercivity and closedness give a minimizer, and strong convexity makes it
unique.  Write \(G_k=\operatorname{grad}F(x_k)\),
\(g_k=\norm{G_k}\), and
\[
  \gamma_k(t)=\operatorname{Exp}_{x_k}(-tG_k),
  \qquad 0\le t\le L_0^{-1}.
\]
If \(g_k=0\), strong convexity gives \(x_k=x_\star\).  Otherwise
\((F\circ\gamma_k)'(0)=-g_k^2<0\).  If this update segment left the current
sublevel, continuity would give a first return
\(t_\star\in(0,L_0^{-1}]\).  Before that return the Hessian is bounded by
\(L_0g\), so the integral Taylor formula would give
\[
\begin{aligned}
  F(\gamma_k(t_\star))
  &\le F(x_k)-t_\star g_k^2
       +\frac{L_0t_\star^2}{2}g_k^2\\
  &\le F(x_k)-\frac{t_\star}{2}g_k^2<F(x_k),
\end{aligned}
\]
contradicting the return.  Hence the complete update stays in the current
sublevel, and the same inequality at \(t=L_0^{-1}\) gives
\[
  F(x_{k+1})\le F(x_k)-\frac{g_k^2}{2L_0}.
\]

Let \(\Delta(x)=F(x)-F(x_\star)\),
\(d=d(x,x_\star)\), and
\(g=\norm{\operatorname{grad}F(x)}\).  Strong convexity at \(x\) and at
\(x_\star\) gives
\[
  \frac{\mu}{2}d^2
  \le\Delta(x)
  \le gd-\frac{\mu}{2}d^2
  \le\frac{g^2}{2\mu}.
\]
The first two expressions also imply \(d\le g/\mu\).  Thus
\(g^2\ge2\mu\Delta\), and the descent inequality yields
\[
  \Delta_{k+1}
  \le\left(1-\frac{\mu}{L_0}\right)\Delta_k.
\]
Iteration proves the value rate; the two preceding inequalities prove the
distance and posterior bounds.  Since
\(g_k^2/(2L_0)\le\Delta_k\), they also give the gradient rate.  Exact
inversion of \(q_{\mu,L}^k\) gives the stated complexity.  If
\(L_0=\mu\), the recurrence gives \(\Delta_1=0\).

For \(F(x)=\mu\|x\|^2/2\) on a Euclidean line, both posterior inequalities
and the \(L_0=\mu\) one-step statement hold with equality.
\end{proof}

For the fiber objective \(F_B(P)=\tfrac12\dai(I,P)^2\),
\cref{lem:airm-curvature-bound,lem:squared-distance-comparison} imply
\begin{equation}
  \operatorname{Hess}F_B\succeq g,
  \qquad
  \operatorname{Hess}F_B(P)\preceq
  \psi(D/\sqrt2)g_P
  \quad\text{if }\dai(I,P)\le D.
  \label{eq:fiber-hessian-bounds}
\end{equation}
The following action instance proves that the constant cannot be decreased
uniformly over dimensions and admissible actions.

\begin{proposition}[Action-family attainment of the radius majorant]
\label{prop:sharp-action-majorant}
Let \(d=3\), \(m=1\), and \(A=B=e_1\).  The action fiber is
\[
  \{\operatorname{diag}(1,\Sigma):\Sigma\in\mathscr P_2\},
\]
its canonical controller is \(I_3\), and for every \(D>0\) there are a point
at AIRM radius \(D\) and a unit fiber tangent \(v\) such that
\begin{equation}
  \operatorname{Hess}F_B(v,v)=\psi(D/\sqrt2).
  \label{eq:sharp-action-majorant-attainment}
\end{equation}
Hence \(\psi(D/\sqrt2)\) is the sharp dimension-uniform radius-only Hessian
majorant within the action family.  Along a radial geodesic in the same fiber,
both posterior inequalities in
\cref{eq:solver-distance-certificate,eq:solver-gap-certificate} hold with
equality.
\end{proposition}

\begin{proof}
Here \(M=1\), \(N=0\), and \(\rho=1\), so the displayed fiber and canonical
controller follow directly from \cref{eq:action-bundle-map}.  The
determinant-one surface \(\mathscr P_2\) is a two-dimensional space of
constant sectional curvature \(-1/2\), as equality holds in the commutator
bound of \cref{lem:airm-curvature-bound}.  On a constant-curvature
\(-\kappa\) plane, the transverse eigenvalue of the Hessian of
\(\tfrac12d(I,\cdot)^2\) at radius \(D\) is
\(D\sqrt\kappa\coth(D\sqrt\kappa)\).  Taking \(\kappa=1/2\) gives
\eqref{eq:sharp-action-majorant-attainment}.
\end{proof}

\subsection{Residual gradient and exact global convergence}

\begin{proof}[Proof of \cref{eq:residual-gradient-identity}]
On the ambient SPD manifold, for
\(\widetilde f(Q)=\tfrac12\dai(R_0,Q)^2\),
\[
  -\operatorname{grad}\widetilde f(Q)=\operatorname{Log}_Q(R_0).
\]
Whitening at \(Q\) turns this tangent into
\(\log(Q^{-1/2}R_0Q^{-1/2})\).  At
\(Q=\mathcal Q_\Sigma\) it is precisely \(\mathcal R(\Sigma)\).
The whitened tangent space of the hidden determinant-one block consists of
\[
  \begin{pmatrix}0&0\\0&Z\end{pmatrix},
  \qquad Z=Z^\top,\quad\tr Z=0.
\]
Since whitening converts the AIRM inner product to the Frobenius inner
product, orthogonal projection retains the \(22\) block and removes its
trace.  Thus the whitened reduced negative gradient is
\(W_B(\Sigma)=\mathcal R_{22}(\Sigma)^\circ\), which proves both identities
in \cref{eq:residual-gradient-identity}.
\end{proof}

\begin{proof}[Proof of \cref{thm:certified-action-solver}]
Write \(f_k=f_B(\Sigma_k)\),
\(g_k=\fro{W_B(\Sigma_k)}\), and
\(D_k=\dai(R_0,\mathcal Q_{\Sigma_k})=\sqrt{2f_k}\).  Trace freeness of
\(W_B(\Sigma_k)\) preserves determinant one, while the exponential preserves
positive definiteness.  The action chart consequently keeps every controller
on \(\Pfiber(B)\).

Let \(\gamma_k(t)\) be the negative-gradient geodesic, parameterized for
\(0\le t\le L_0^{-1}\), and put
\(\varphi_k=f_B\circ\gamma_k\).  If \(g_k=0\), the iterate is optimal.
Otherwise \(\varphi_k'(0)=-g_k^2<0\).  If the update segment left the
current objective sublevel, there would be a first return
\(t_\star\in(0,L_0^{-1}]\).  Before that return,
\(D(\gamma_k(t))\le D_k\le D_0\), so
\eqref{eq:fiber-hessian-bounds} bounds the Hessian by \(L_0g\).  The
integral Taylor formula would then imply
\[
  \varphi_k(t_\star)
  \le\varphi_k(0)-t_\star g_k^2
      +\frac{L_0t_\star^2}{2}g_k^2
  <\varphi_k(0),
\]
a contradiction.  Thus the complete update segment stays in the current
sublevel.  Applying the same Taylor bound at \(L_0^{-1}\) gives
\[
  f_{k+1}\le f_k-\frac1{L_0}g_k^2
  +\frac1{2L_0}g_k^2
  =f_k-\frac1{2L_0}g_k^2.
\]
Thus \(D_{k+1}\le D_k\), closing the induction and proving
\cref{eq:solver-descent} globally.

Let \(\Delta(\Sigma)=f_B(\Sigma)-f_B(\Sigma_\star)\).  Strong convexity
and Cauchy--Schwarz give, with
\(d=\dai(\Sigma,\Sigma_\star)\) and
\(g=\norm{\operatorname{grad}f_B(\Sigma)}\),
\begin{equation}
  \Delta(\Sigma)\le gd-\frac12d^2\le\frac12g^2.
  \label{eq:proof-pl}
\end{equation}
Hence \(g^2\ge2\Delta\), and the descent inequality yields
\[
  \Delta_{k+1}\le(1-L_0^{-1})\Delta_k=q_0\Delta_k.
\]
Since \(\Delta_0\le f_0=D_0^2/2\), this proves
\cref{eq:solver-value-rate}.  Strong convexity based at the minimizer gives
\(d^2\le2\Delta\), proving \cref{eq:solver-distance-rate}.  Combining
\(f_{k+1}\ge f_B(\Sigma_\star)\) with the descent inequality gives
\(g_k^2\le2L_0\Delta_k\), proving
\cref{eq:solver-residual-rate}.

The two strong-convexity inequalities also give
\(\Delta\ge d^2/2\).  Together with the first bound in
\eqref{eq:proof-pl}, this implies \(d\le g\), including the case \(d=0\),
while the second bound in \eqref{eq:proof-pl} gives
\(\Delta\le g^2/2\).  These are
\cref{eq:solver-distance-certificate,eq:solver-gap-certificate}.
Finally, for \(D_0>0\),
\(q_0^k=e^{-k\chi_0}\).  Solving the three rate bounds exactly for \(k\)
gives
\cref{eq:solver-value-complexity,eq:solver-distance-complexity,eq:solver-residual-complexity}.
The case \(D_0=0\) is immediate.
\end{proof}

\subsection{Active and inexact implementations}

\begin{proof}[Proof of \cref{cor:active-core-solver}]
Retain \(A=UC\) from \cref{sec:action-bundle}.  Choose \(V_a\) so that
\(E=[U,V_a]\) is an orthonormal basis of \(\mathcal W_A\).  Then
\[
  BC^{-1}=UM+V_aN_a,\qquad
  T_a=N_aM^{-1},\qquad
  L_a=\begin{pmatrix}I&0\\T_a&I\end{pmatrix},
  \qquad
  R_{0,a}=L_a^{-1}L_a^{-\top}.
\]
Complete \(E\) by \(Z\) when \(q>0\).  In the basis \([U,V_a,Z]\), the
full action coordinate satisfies
\[
  N=\begin{pmatrix}N_a\\0\end{pmatrix},
  \qquad
  L=L_a\oplus I_q,
  \qquad
  R_0=R_{0,a}\oplus I_q.
\]
The initializer in the corollary is therefore the full hidden state
\[
  \Sigma_0=I_s\oplus I_q.
\]

More generally, suppose
\[
  \Sigma=\Sigma_a\oplus\beta I_q,
  \qquad
  \det\Sigma_a\,\beta^q=1.
\]
Then
\[
  \mathcal Q_a=\operatorname{diag}(M,\rho\Sigma_a),
  \qquad
  \mathcal Q_\Sigma
  =\mathcal Q_a\oplus(\rho\beta)I_q,
\]
and block functional calculus gives
\[
  \mathcal R_a
  =\log(\mathcal Q_a^{-1/2}R_{0,a}\mathcal Q_a^{-1/2}),
  \qquad
  \mathcal R(\Sigma)
  =\mathcal R_a\oplus r_cI_q,
  \qquad
  r_c=-\log(\rho\beta).
\]
The hidden trace-free projection subtracts the common mean
\[
  \tau=\frac{\tr\mathcal R_{22}^a+qr_c}{s+q}.
\]
Consequently,
\begin{equation}
  W_B(\Sigma)=W_a\oplus w_cI_q,
  \quad
  W_a=\mathcal R_{22}^a-\tau I_s,
  \quad
  w_c=r_c-\tau,
  \qquad
  \fro{W_B(\Sigma)}^2=\fro{W_a}^2+qw_c^2.
\label{eq:active-residual-splitting}
\end{equation}
The full exponential update now splits exactly into
\begin{equation}
  \Sigma_a^+
  =\Sigma_a^{1/2}e^{W_a/L_0}\Sigma_a^{1/2},
  \qquad
  \beta^+=\beta e^{w_c/L_0}.
\label{eq:active-solver-update-proof}
\end{equation}
Since \(\tr W_a+qw_c=0\), this update preserves
\(\det\Sigma_a\,\beta^q=1\).  The action-bundle map then preserves
positive definiteness, determinant one, and \(PA=B\).  This proves
\eqref{eq:active-residual-splitting} and
\eqref{eq:active-solver-update-proof}; the
stopping certificate follows from
\cref{eq:solver-distance-certificate}.

A rank-revealing thin QR of the \(d\times2m\) input costs \(O(dm^2)\).
All remaining preprocessing acts on matrices of order at most \(r\), costing
\(O(r^3)\).  Each iteration uses one \(r\times r\) logarithm and inverse
square root and one \(s\times s\) exponential, hence \(O(r^3)\) arithmetic
and \(O(r^2)\) working memory.  Combining this cost with
\cref{eq:solver-residual-complexity} gives
\cref{eq:active-total-complexity}.  The representation
\[
  Px=cx+E(H-cI_r)E^\top x
\]
costs \(O(dr+r^2)\) per application.  If \(q=0\), every complement term is
absent and the same construction is the full solver.  If \(s=0\), the
determinant condition forces \(\beta=1\), leaving no variable.
\end{proof}

\begin{proof}[Proof of \cref{cor:inexact-action-solver}]
Put \(W_k=W_B(\Sigma_k)\), \(g_k=\fro{W_k}\), and
\(\widetilde W_k=W_k+E_k\).  First,
\cref{eq:computable-relative-error-test} and \(b_k\ge\fro{E_k}\) give
\[
  \fro{E_k}
  \le\frac{\delta}{1+\delta}
      \left(\fro{W_k}+\fro{E_k}\right),
\]
so rearrangement proves \cref{eq:relative-residual-error}.  Consequently,
\[
  (1-\delta)g_k\le\fro{\widetilde W_k}
  \le(1+\delta)g_k,
\qquad
  \ip{-W_k}{\widetilde W_k}_F
  \le-(1-\delta)g_k^2.
\]
If \(g_k=0\), the relative error forces a zero update.  Otherwise consider
the inexact update geodesic.  If it left the current objective sublevel,
there would be a first return \(t_\star\in(0,\eta]\).  Before that return,
the identity radius is at most \(D_k\le D_0\le\widehat D_0\), so the Hessian
is bounded by \(\widehat L_0g\).  The integral Taylor bound would give
\[
  f_B(\gamma_k(t_\star))
  \le f_k-t_\star(1-\delta)g_k^2
      +\frac{\widehat L_0t_\star^2}{2}(1+\delta)^2g_k^2
  <f_k,
\]
contradicting the return.  Hence the whole segment remains in the current
sublevel.  The descent lemma at \(t=\eta\) implies
\[
  f_{k+1}
  \le f_k-\eta(1-\delta)g_k^2
  +\frac{\widehat L_0\eta^2}{2}(1+\delta)^2g_k^2
  =f_k-\frac{c_\delta}{2\widehat L_0}g_k^2.
\]
Hence \(D_{k+1}\le D_k\).  Applying the global PL inequality from the exact
proof gives
\[
  \Delta_{k+1}\le\widehat q_\delta\Delta_k.
\]
Iteration and
\(\Delta_0\le D_0^2/2\le\widehat D_0^2/2\) prove
\cref{eq:inexact-value-rate}; strong convexity gives
\cref{eq:inexact-distance-rate}.  The descent bound and
\(f_{k+1}\ge f_B(\Sigma_\star)\) imply
\[
  g_k^2\le\frac{2\widehat L_0}{c_\delta}\Delta_k.
\]
Together with \(\fro{\widetilde W_k}\le(1+\delta)g_k\), the value rate, and
\(\sqrt{c_\delta}=(1-\delta)/(1+\delta)\), this proves
\cref{eq:inexact-residual-rate}.  Finally, the exact posterior certificates and
\(g_k\le\fro{\widetilde W_k}/(1-\delta)\) give
\cref{eq:inexact-certificates}.
\end{proof}

\subsection{Nested recovery}

\begin{proof}[Proof of \cref{thm:nested-multisecant}]
For every \(k\),
\[
  A_k^\top B_k
  =c_H(S^{(k)})^\top HS^{(k)}\succ0,
\]
so \(\mathscr C_k\) is nonempty, closed, and totally geodesic by
\cref{thm:action-bundle}.  Also
\(P_\star A_k=c_HH^{-1}HS^{(k)}=B_k\), so
\(P_\star\in\mathscr C_k\).  Nested column spaces imply
\(\mathscr C_{k+1}\subset\mathscr C_k\).  Metric projection onto a
nonempty closed geodesically convex subset of a Hadamard manifold exists
uniquely, proving that every update is well defined and retains every earlier
action.

For completeness, if \(y=\operatorname{Proj}_C(x)\) and \(z\in C\), first
variation along the geodesic from \(y\) to \(z\) gives
\(\ip{\operatorname{Log}_y(x)}{\operatorname{Log}_y(z)}\le0\).  Since the
exponential map is distance expanding,
\[
\begin{aligned}
  d(x,z)^2
  &\ge\norm{\operatorname{Log}_y(x)-\operatorname{Log}_y(z)}^2\\
  &\ge d(x,y)^2+d(y,z)^2.
\end{aligned}
\]
Apply this with \(x=P_{k-1}\), \(y=P_k\), \(z=P_\star\), and
\(C=\mathscr C_k\) to obtain
\cref{eq:multisecant-pythagorean}.  Summation telescopes to
\cref{eq:multisecant-energy}; Cauchy--Schwarz then gives
\cref{eq:multisecant-length}.

If \(r_K=d-1\), the bundle fiber is diffeomorphic to
\(\mathscr P_{d-r_K}=\mathscr P_1=\{1\}\), so
\(\mathscr C_K=\{P_\star\}\).  If \(r\le d-2\), then
\(n=d-r\ge2\) and the fiber is diffeomorphic to \(\mathscr P_n\), whose
dimension \(n(n+1)/2-1\) is positive; it therefore contains infinitely many
completions.  This proves the sharp threshold and the stated block count.

Finally, for the \(k\)-th projection set
\[
  \widetilde P=P_{k-1}^{-1/2}PP_{k-1}^{-1/2},\quad
  \widetilde A_k=P_{k-1}^{1/2}A_k,\quad
  \widetilde B_k=P_{k-1}^{-1/2}B_k.
\]
Then \(PA_k=B_k\) is equivalent to
\(\widetilde P\widetilde A_k=\widetilde B_k\), determinant one is
preserved, and congruence invariance turns the objective into
\(\dai(I,\widetilde P)\).  Thus every outer update is an identity-reference
action problem.  The exact projection is the limit covered by
\cref{thm:certified-action-solver}; \cref{cor:inexact-action-solver} certifies
finite inner approximations but is not used in the exact Pythagorean identities.
\end{proof}

\section{Closest-Work Comparison}
\label{app:closest-comparison}

The solver theorem combines standard Hadamard first-order principles with an
action-specific geometric construction.  \Cref{tab:theorem-level-comparison}
separates those responsibilities.  It compares theorem outputs, not algorithms
with different feasible models.

\begin{table}[ht]
\centering
\begingroup
\setlength{\tabcolsep}{5pt}
\footnotesize
\begin{tabular}{@{}p{0.22\textwidth}p{0.34\textwidth}p{0.38\textwidth}@{}}
\toprule
Work & Established output & Action-specific distinction \\
\midrule
Hadamard first-order methods \citep{zhang2016firstorder}
& Global complexity for smooth or nonsmooth geodesically convex objectives,
including strongly convex and curvature-dependent regimes
& No parameterization of \(PA=B\), logarithmic action residual, or active
controller realization \\

Hadamard convex analysis \citep{bacak2014convex}
& Projection, proximal, and convex-analytic foundations on closed geodesically
convex sets
& Supplies abstract projection principles; does not construct or solve the
action fiber \\

Definite matrix equations and geodesic SPD projection
\citep{tian2013hermitian,lim2004best,tumpach2024totally,sra2015conic}
& Algebraic definite-solution identities, projection results, total-geodesy
criteria, and algorithms for geodesic SPD or conic models
& No global determinant-one action bundle together with its residual,
active realization, and action posterior certificate \\

Least-change and multi-secant updates
\citep{dennis1979leastchange,schnabel1983multiple,fletcher1991variational,
gao2018block,boutet2022grouped,kanamori2013bregman}
& Update formulas and convergence properties for secant equations under
method-specific norms or divergences
& Do not parameterize the complete AIRM action fiber or prove the normalized
rank-\(d-1\) identification threshold \\

This paper
& For \(PA=B\), \(P\succ0\), \(\det P=1\): global action bundle, total
geodesy, analytic canonical section, exact
residual, global rate, posterior certificates, and active total complexity
& Specializes the general principles into an exact action solver with
posterior certificates and conditional residual-error propagation \\
\bottomrule
\end{tabular}
\endgroup
\caption{Theorem-level comparison with the closest geometric optimization and
multi-secant baselines.}
\label{tab:theorem-level-comparison}
\end{table}

The conditional inexact corollary assumes rigorous matrix-function error and
radius majorants.  It propagates those certificates through the geometric
solver; unlike a certified numerical linear-algebra package, it does not
construct the majorants or establish bit complexity.

\end{document}